\documentclass[a4paper,10pt]{elsarticle}

\def\grayscale{}
\newif\ifusetikzexternalize
\usetikzexternalizefalse

\usepackage{amsmath,amsthm,amssymb,scalerel}
\usepackage{bm}
\usepackage{color}
\usepackage{mathtools}

\usepackage{thmtools,thm-restate}
\declaretheorem{example}

\declaretheorem{lemma}

\newcommand{\+}[1]{\boldsymbol{\ensuremath{\mathbf{#1}}}}

\newcommand\mtext[1]{\text{#1}}
\def\volfrac{\alpha}
\newcommand\symmdiff{\triangle}
\newcommand\halfspace[1]{l\roundpar{#1}}

\newcommand\roundpar[1]{\left( #1 \right)}
\newcommand\squarepar[1]{\left[ #1 \right]}
\newcommand\curlypar[1]{\left\lbrace #1 \right\rbrace}

\newcommand\oneDIntegral[4]{\int_{#1}^{#2} #3 \hspace{1ex} d#4}
\newcommand\integral[3]{\int_{#1} #2 \hspace{1ex} d#3}

\newcommand\defeq{\mathrel{\mathop:}=}

\newcommand\half{\frac{1}{2}}

\newcommand\dt{\delta}

\usepackage{stackengine}
\DeclareFontFamily{U}{mathx}{\hyphenchar\font45}
\DeclareFontShape{U}{mathx}{m}{n}{
      <5> <6> <7> <8> <9> <10>
      <10.95> <12> <14.4> <17.28> <20.74> <24.88>
      mathx10
      }{}
\DeclareSymbolFont{mathx}{U}{mathx}{m}{n}
\DeclareMathAccent{\widecheck}    {0}{mathx}{"71}
\newcommand\approximate[1]{\widecheck{#1}}

\newcommand\preimage{P}

\newcommand\gradient{\nabla}

\newcommand\abs[1]{\left\lvert #1 \right\rvert}
\newcommand\norm[1]{\left\lVert #1 \right\rVert}

\newcommand\weber{\text{We}}

\newcommand\ratio[1]{\mathcal{R}_{#1}}

\DeclareMathOperator{\atan}{atan}

\newcommand\set[2]{
  \{\,#1 \mid #2\,\}
}

\newcommand\flowmap[2]{\Psi^{#1}}

\usepackage{siunitx}

\DeclareMathOperator*{\argmin}{arg\,min}

\usepackage{geometry}
\usepackage{graphicx}
\usepackage{multirow}

\usepackage{tikz}
\usepackage[utf8]{inputenc}
\usepackage{pgfplots}
\usepackage{pgfgantt}
\usepackage{pdflscape}
\ifusetikzexternalize

  \pgfplotsset{compat=newest}
  \pgfplotsset{plot coordinates/math parser=false}
  \pgfplotsset{
      legend image with text/.style={
          legend image code/.code={%
              \node[anchor=center] at (0.3cm,0cm) {#1};
          }
      },
  }
  \usetikzlibrary{
      matrix,
  }
  \pgfplotsset{
      compat=1.3,
  }
  \usepgfplotslibrary{fillbetween}
  \usetikzlibrary{
    arrows.meta,
    external}
  \tikzset{external/system call={pdflatex \tikzexternalcheckshellescape -halt-on-error
    -interaction=batchmode -jobname "\image" "\texsource" && % or ;
    pdfseparate -f 1 -l 1 "\image".pdf "\image"_tmp.pdf &&
    mv "\image"_tmp.pdf "\image".pdf}}
  \newcommand\inputtikzorpdf[1]{
    \tikzsetnextfilename{#1\grayscale}
    \input{./tikz/#1\grayscale}
  }
\else

  \newcommand\inputtikzorpdf[1]{\includegraphics{pdf/#1\grayscale}}
  \newcommand\tikzexternalenable{}
  \newcommand\tikzexternaldisable{}
\fi

\usepackage{subcaption}
\def\onefigwidth{\textwidth}
\def\twofigwidth{0.475\textwidth}
\def\threefigwidth{0.3\textwidth}
\usepackage{natbib}
\usepackage{booktabs}
\usepackage[hidelinks]{hyperref}
\usepackage{cleveref}
\usepackage{autonum} % <- label only those equations that are referenced
\usepackage{import}

\usepackage{enumitem}
\newtheorem{step}{Steps}
\crefname{step}{step}{steps}
\newlist{stepenum}{enumerate}{1} % also creates a counter called 'stepenumi'
\setlist[stepenum]{label=(\alph*), ref=(\alph*)}
\crefalias{stepenumi}{step}

\crefformat{appendix}{#2#1#3}

\begin{document}
  \begin{frontmatter}

  \author{
    Ronald A. Remmerswaal
  }
  \author{
    Arthur E.P. Veldman
  }

  \address{
    Bernoulli Institute, University of Groningen\\
    PO Box 407, 9700 AK Groningen, The Netherlands
  }

  \begin{abstract}
    For capillary driven flow the interface curvature is essential in the modelling of surface tension via the imposition of the Young--Laplace jump condition.
We show that traditional geometric volume of fluid (VOF) methods, that are based on a piecewise linear approximation of the interface, do not lead to an interface curvature which is convergent under mesh refinement in time-dependent problems.
Instead, we propose to use a piecewise parabolic approximation of the interface, resulting in a class of piecewise parabolic interface calculation (PPIC) methods.
In particular, we introduce the parabolic LVIRA and MOF methods, PLVIRA and PMOF, respectively.
We show that a Lagrangian remapping method is sufficiently accurate for the advection of such a parabolic interface.

It is numerically demonstrated that the newly proposed PPIC methods result in an increase of reconstruction accuracy by one order, convergence of the interface curvature in time-dependent advection problems and Weber number independent convergence of a droplet translation problem, where the advection method is coupled to a two-phase Navier--Stokes solver.
The PLVIRA method is applied to the simulation of a 2D rising bubble, which shows good agreement to a reference solution.
  \end{abstract}

  \title{Parabolic interface reconstruction for 2D volume of fluid methods}

  \begin{keyword}
    two-phase flow \sep volume of fluid method \sep parabolic reconstruction
  \end{keyword}
\end{frontmatter}

\section{Introduction}\label{sec:introduction}
The advection of the phase interface plays a central role in the simulation of immiscible, and in our case incompressible, two-phase flow.
Therefore, much research has been performed towards the development of accurate, efficient and robust interface advection methods.
Explicit modelling of the interface~\citep{Tryggvason2001} results in a highly accurate and continuous representation, for which the interface resolution is essentially unrelated to the resolution of the mesh.
Changes in topology can be taken into account~\citep{Bo2011}, but this is nontrivial.

The interface can also be represented implicitly, either using a level set or a volume fraction function, resulting in level set and volume of fluid (VOF) methods respectively.
Using level sets~\citep{Osher1988,Gibou2017} results in an efficient method for which changes in topology are automatically taken into account.
Level set methods however do not inherently conserve mass and require artificial redistancing to ensure that the level set function remains a distance function~\citep{Sussman1994}.
On the contrary, VOF methods (see for instance the works of~\citet{Hirt1981,Youngs1982,Rider1997,Puckett1997,Rudman1998,Harvie2001,Lopez2004,Weymouth2010}) can inherently conserve mass and result, in particular if geometric VOF methods are considered, in a sharper interface.
Geometric VOF methods rely on the geometric reconstruction of the interface which results in a sharp interface representation (see e.g.~\cref{fig:intro:flower_shape_MoF_sd_2}) thus reflecting the assumed immiscibility of the fluids.
On the other hand, algebraic VOF methods lack a geometric reconstruction of the interface, and result in a diffuse interface.
In this paper we will focus our attention on geometric VOF methods that are based on Lagrangian remapping; this is discussed in~\cref{sec:remap}.

Level set methods are sometimes preferred for their `smoother' interface representation, thus resulting in the ability to easily approximate the local interface geometry, such as the interface normal vector as well as the interface curvature.
In our experience however, a sufficiently accurate volume fraction field equally provides the ability to compute the interface normal vector as well as the interface curvature at high accuracy, the latter using local height-functions (LHFs).
What we exactly mean by `sufficiently accurate' is discussed in~\cref{sec:curvature}.

\begin{figure}
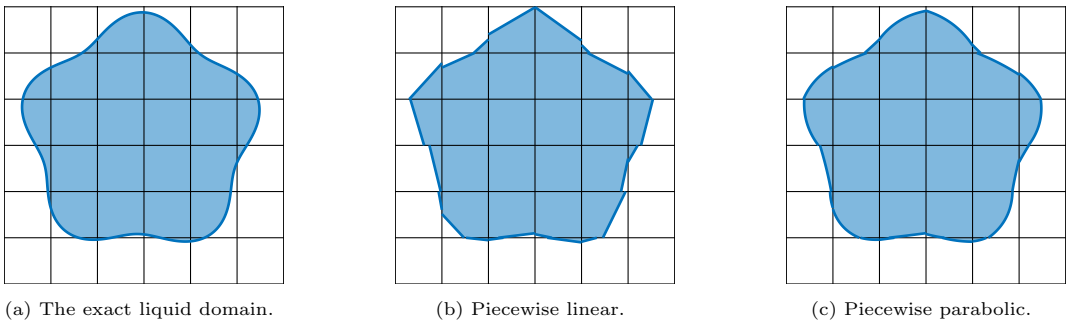

  \subcaptionbox{The exact liquid domain.
  \label{fig:intro:flower_shape_sd_1}}
  [\threefigwidth]{
    \def\tikzWidth{\textwidth*0.25}
    \def\tikzHeight{\textwidth*0.25}
    \inputtikzorpdf{intro_flower_shape_sd_1}
  }\hfill
  \subcaptionbox{Piecewise linear.
  \label{fig:intro:flower_shape_MoF_sd_2}}
  [\threefigwidth]{
    \def\tikzWidth{\textwidth*0.25}
    \def\tikzHeight{\textwidth*0.25}
    \inputtikzorpdf{intro_flower_shape_MoF_sd_2}
  }\hfill
  \subcaptionbox{Piecewise parabolic.
  \label{fig:intro:flower_shape_PMoF_sd_2}}
  [\threefigwidth]{
    \def\tikzWidth{\textwidth*0.25}
    \def\tikzHeight{\textwidth*0.25}
    \inputtikzorpdf{intro_flower_shape_PMoF_sd_2}
  }
  \caption{Example of a liquid domain and two different approximations (using the MOF and PMOF method respectively).
  Here the exact interface shape was defined as the zero level set of~\cref{eqn:results:recon:levelset}.}
  \label{fig:intro:flower_shape}
\end{figure}
For capillary driven flow the interface curvature is essential in the modelling of surface tension via the imposition of the Young--Laplace jump condition.
When traditional geometric reconstruction methods, that are based on a piecewise linear approximation of the interface (see~\cref{fig:intro:flower_shape_MoF_sd_2} and~\cref{sec:plic}), are used, we find that the curvature from LHFs does not converge under mesh refinement for time-dependent problems.
Nonconvergence of the interface curvature results in spurious and unphysical currents, as discussed in~\citep{Magnini2016}.
A theoretical understanding is provided in~\cref{sec:remap,sec:curvature} where we analyse the accuracy of the advection method and of the LHF-based curvature respectively.
These results are confirmed numerically in~\cref{sec:results:reverse,sec:results:translation}.
Based on this observed lack of convergence, we propose to use a piecewise parabolic approximation of the interface instead, as shown in~\cref{fig:intro:flower_shape_PMoF_sd_2} and discussed in~\cref{sec:parabolic}.
This essentially includes the curvature in the interface advection, and in~\cref{sec:results} we will demonstrate that this leads to convergence of the curvature, also for time-dependent problems.

Throughout this paper we will denote the positive time-step by $\delta$, and the maximum diameter of all control volumes by $h$.
Furthermore, a numerical approximation to some quantity $y$ is denoted by $\approximate{y}(\delta, h)$ where we never explicitly indicate the dependence on $\delta, h$.
The statement `$y$ converges' is understood as $\approximate{y}(\delta, h)$ converges to $y$ under mesh refinement
\begin{equation}
  \lim_{h \rightarrow 0} \approximate{y}(\delta, h) = y,
\end{equation}
where the time-step $\delta$ satisfies the CFL time-step restriction
\begin{equation}\label{eqn:intro:cfl}
  \delta \le \frac{h}{U},
\end{equation}
such that $h\rightarrow 0$ implies $\delta \rightarrow 0$.

\section{Advection methods based on Lagrangian remapping}\label{sec:remap}
In this section we discuss the advection method that will be used to approximately track the interface $I(t)$, and show that this advection method is indeed sufficiently accurate for the advection of a parabolic interface approximation, as we propose in~\cref{sec:parabolic}.

We denote the computational domain by $\Omega \subset \mathbb{R}^2$, the interface between the two fluids by $I(t)$ and the corresponding liquid domain by $\Omega^l(t) \subset \Omega$.

The zeroth and first moment of some set $A \subset \mathbb{R}^2$ will be denoted by
\begin{equation}
  M_0(A) \defeq \integral{A}{}{V}, \quad \+M_1(A) \defeq \integral{A}{\+x}{V},
\end{equation}
respectively (the first moment is introduced for use in the MOF reconstruction method discussed in~\cref{sec:plic:mof}).
Given the zeroth and first moment, the centroid is given by
\begin{equation}
  \+C(A) \defeq \frac{\+M_1(A)}{M_0(A)}.
\end{equation}
The liquid volume $M_{0,c}^l$ and its control volume fraction $\volfrac_c \in [0, 1]$ are then defined as
\begin{equation}
  M_{0,c}^l \defeq M_0(c \cap \Omega^l), \quad \volfrac_c \defeq \frac{M_{0,c}^l}{M_0(c)},
\end{equation}
for some control volume $c \subset \Omega$.
The liquid first moment and centroid are similarly denoted by $\+M^l_{1,c}$ and $\+C^l_{c}$ respectively.
% From the mass conservation equation we can derive the following conservation equation of the liquid volume
% \begin{equation}\label{eqn:intro:gvof_eqn}
%   \frac{d}{dt} M_{0,c}^l + \integral{\partial (c \cap \Omega^l(t)) \setminus I(t)}{\+u \cdot \+\eta}{S} = 0.
% \end{equation}

\subsection{Lagrangian remapping}
\begin{figure}
  \subcaptionbox{The dashed curves correspond to the boundary of the preimage of $c$, which is denoted by $\Psi^{-\delta} c$.
  The solid curve corresponds to the fluid interface $I^{(n)}$ and the shaded region corresponds to $\Psi^{-\delta} c \cap \Omega^{l,(n)}$ (cf. the right-hand side of~\cref{eqn:remap:remap_m0}).
  \label{fig:remap:remapping_oldtime}}
  [\twofigwidth]{
    %% Creator: Inkscape inkscape 0.92.5, www.inkscape.org
%% Accompanies image file 'remapping.pdf' (pdf, eps, ps)
\begingroup%
  \makeatletter%
  \providecommand\color[2][]{%
    \errmessage{(Inkscape) Color is used for the text in Inkscape, but the package 'color.sty' is not loaded}%
    \renewcommand\color[2][]{}%
  }%
  \providecommand\transparent[1]{%
    \errmessage{(Inkscape) Transparency is used (non-zero) for the text in Inkscape, but the package 'transparent.sty' is not loaded}%
    \renewcommand\transparent[1]{}%
  }%
  \providecommand\rotatebox[2]{#2}%
  \newcommand*\fsize{\dimexpr\f@size pt\relax}%
  \newcommand*\lineheight[1]{\fontsize{\fsize}{#1\fsize}\selectfont}%
  \ifx\svgwidth\undefined%
    \setlength{\unitlength}{143.35456348bp}%
    \ifx\svgscale\undefined%
      \relax%
    \else%
      \setlength{\unitlength}{\unitlength * \real{\svgscale}}%
    \fi%
  \else%
    \setlength{\unitlength}{\svgwidth}%
  \fi%
  \global\let\svgwidth\undefined%
  \global\let\svgscale\undefined%
  \makeatother%
  \begin{picture}(1,0.82245924)%
    \lineheight{1}%
    \setlength\tabcolsep{0pt}%
    \put(0,0){\includegraphics[width=\unitlength,page=1]{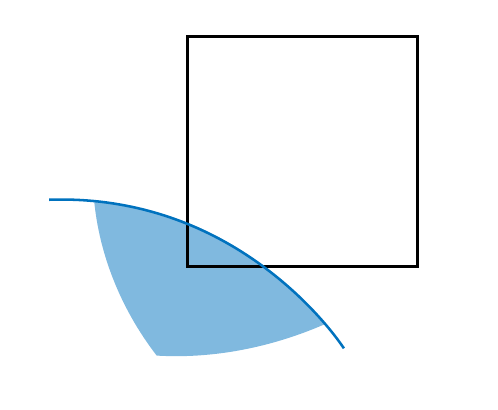}}%
    \put(0.43375673,0.52457867){\color[rgb]{0.11372549,0.11372549,0.11372549}\makebox(0,0)[lt]{\lineheight{1.25}\smash{\begin{tabular}[t]{l}$\flowmap{-\dt}{}c$\end{tabular}}}}%
    \put(0,0){\includegraphics[width=\unitlength,page=2]{remapping.pdf}}%
    \put(0.69879181,0.06868602){\color[rgb]{0,0.44705882,0.74509804}\makebox(0,0)[lt]{\lineheight{1.25}\smash{\begin{tabular}[t]{l}$I^{(n)}$\end{tabular}}}}%
    \put(0,0){\includegraphics[width=\unitlength,page=3]{remapping.pdf}}%
    \put(0.00427254,0.02290812){\color[rgb]{0.11372549,0.11372549,0.11372549}\makebox(0,0)[lt]{\lineheight{1.25}\smash{\begin{tabular}[t]{l}$\flowmap{-\dt}{t^{(n+1)}} c \cap \Omega^{l,(n)}$\end{tabular}}}}%
  \end{picture}%
\endgroup%

  }\hfill
  \subcaptionbox{The square corresponds to the boundary of the control volume $c$.
  The solid curve corresponds to the fluid interface $I^{(n+1)}$ and the shaded region corresponds to $c \cap \Omega^{l,(n+1)}$ (cf. the left-hand side of~\cref{eqn:remap:remap_m0}).
  \label{fig:remap:remapping_newtime}}
  [\twofigwidth]{
    %% Creator: Inkscape inkscape 0.92.5, www.inkscape.org
%% Accompanies image file 'remapping_newtime.pdf' (pdf, eps, ps)
\begingroup%
  \makeatletter%
  \providecommand\color[2][]{%
    \errmessage{(Inkscape) Color is used for the text in Inkscape, but the package 'color.sty' is not loaded}%
    \renewcommand\color[2][]{}%
  }%
  \providecommand\transparent[1]{%
    \errmessage{(Inkscape) Transparency is used (non-zero) for the text in Inkscape, but the package 'transparent.sty' is not loaded}%
    \renewcommand\transparent[1]{}%
  }%
  \providecommand\rotatebox[2]{#2}%
  \newcommand*\fsize{\dimexpr\f@size pt\relax}%
  \newcommand*\lineheight[1]{\fontsize{\fsize}{#1\fsize}\selectfont}%
  \ifx\svgwidth\undefined%
    \setlength{\unitlength}{143.35456348bp}%
    \ifx\svgscale\undefined%
      \relax%
    \else%
      \setlength{\unitlength}{\unitlength * \real{\svgscale}}%
    \fi%
  \else%
    \setlength{\unitlength}{\svgwidth}%
  \fi%
  \global\let\svgwidth\undefined%
  \global\let\svgscale\undefined%
  \makeatother%
  \begin{picture}(1,0.82245924)%
    \lineheight{1}%
    \setlength\tabcolsep{0pt}%
    \put(0,0){\includegraphics[width=\unitlength,page=1]{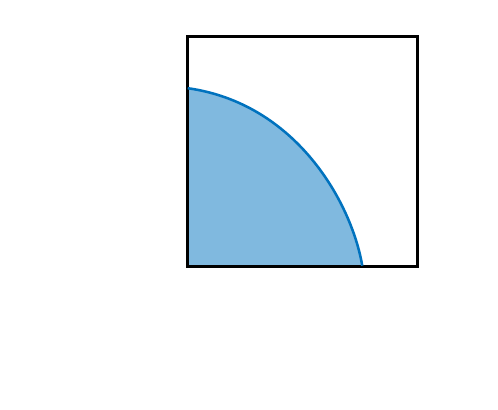}}%
    \put(0.72241448,0.20981343){\color[rgb]{0,0.44705882,0.74509804}\makebox(0,0)[lt]{\lineheight{1.25}\smash{\begin{tabular}[t]{l}$I^{(n+1)}$\end{tabular}}}}%
    \put(0,0){\includegraphics[width=\unitlength,page=2]{remapping_newtime.pdf}}%
    \put(0.107225,0.19429706){\color[rgb]{0.11372549,0.11372549,0.11372549}\makebox(0,0)[lt]{\lineheight{1.25}\smash{\begin{tabular}[t]{l}$c \cap \Omega^{l,(n+1)}$\end{tabular}}}}%
    \put(0,0){\includegraphics[width=\unitlength,page=3]{remapping_newtime.pdf}}%
  \end{picture}%
\endgroup%

  }
  \caption{Illustration of a Lagrangian remapping step, see also~\cref{eqn:remap:remap,eqn:remap:remap_m0}.}
  \label{fig:remap:remapping}
\end{figure}
We let $\Psi^{\delta}\+x_0$ denote the flow map of the velocity field $\+u$ over the time interval $[t^{(n)}, t^{(n+1)}]$, that is, $\Psi^{\delta}$ solves the initial value problem
\begin{equation}
  \frac{d}{dt}\+x = \+u(t, \+x(t)), \quad \+x(t^{(n)}) = \+x_0,
\end{equation}
resulting in $\Psi^{\delta}\+x_0 = \+x(t^{(n+1)})$.
The action of the flow map naturally extends to the application on sets of positions, such as $\Psi^{\delta} c$.
The preimage of the control volume $c$ under the flow map, which we denote by $\Psi^{-\delta} c$, is the set of points which end up inside $c$ under the flow map $\Psi^{\delta}$, and will henceforth be referred to as `the preimage of $c$'.
See also~\cref{fig:remap:remapping}.

Note that since the interface $I(t)$ is advected with the velocity field $\+u$, it follows that
\begin{equation}
  I^{(n+1)} = \Psi^\delta I^{(n)} \quad\Rightarrow\quad \Omega^{l,(n+1)} = \Psi^\delta \Omega^{l,(n)}.
\end{equation}
By the invertibility of $\Psi^{\delta}$ it holds that $\Psi^{\delta}(A \cap B) = \Psi^{\delta}A \cap \Psi^{\delta}B$ and therefore
\begin{equation}\label{eqn:remap:remap}
  c \cap \Omega^{l,(n+1)} = \Psi^{\delta}(\Psi^{-\delta}c \cap \Omega^{l,(n)}).
\end{equation}
Computing the zeroth moment of~\cref{eqn:remap:remap}, and by making use of the fact that the velocity field is divergence free\footnote{Throughout this article it will be assumed that the velocity field is divergence free, this is however not an inherent limitation of our proposed method.} (which implies that the flow map is area preserving, i.e. $M_0(\Psi^\delta A) = M_0(A)$) results in
\begin{equation}\label{eqn:remap:remap_m0}
  M_{0,c}^{l,(n+1)} = M_0\roundpar{\Psi^{-\delta} c \cap \Omega^{l,(n)}}.
\end{equation}
Equation~\eqref{eqn:remap:remap_m0} is key to understanding how a Lagrangian remapping-based geometric VOF method can be constructed, since such a method approximates each of the terms on the right-hand side of~\cref{eqn:remap:remap_m0}:
\begin{itemize}
  \item The liquid domain $\Omega^{l,(n)}$ is approximated per control volume.
  This will be discussed in~\cref{sec:plic,sec:parabolic} for the piecewise linear and piecewise parabolic approximation of the interface respectively.
  \item The preimage $\Psi^{-\delta} c$ is approximated by a polygon whose vertices are approximated using numerical integration, leading to two approximation errors.
  This is what we will discuss next.
\end{itemize}

\subsection{Approximate Lagrangian remapping}\label{sec:remap:method}
We denote the exact preimage of the control volume $c$ by
\begin{equation}
  \preimage_c \defeq \Psi^{-\delta} c.
\end{equation}
The approximate preimage $\approximate{\preimage_c}$ is constructed by making two approximations.
First we approximate the preimage by a polygonal representation denoted by $\preimage^\mtext{Rep}_c$, which is defined by connecting the corners of the exact preimage $\preimage_c$ by straight line segments, as is shown in~\cref{fig:remap:remapping_errors}.
\begin{figure}
  \subcaptionbox{The area of the hatched pattern corresponds to the reconstruction error $E_0^\mtext{Rec}$ ~\cref{eqn:remap:error:rec}.
  \label{fig:remap:remapping_errors_rec}}
  [\threefigwidth]{
    %% Creator: Inkscape inkscape 0.92.5, www.inkscape.org
%% Accompanies image file 'remapping_errors_rec.pdf' (pdf, eps, ps)
\begingroup%
  \makeatletter%
  \providecommand\color[2][]{%
    \errmessage{(Inkscape) Color is used for the text in Inkscape, but the package 'color.sty' is not loaded}%
    \renewcommand\color[2][]{}%
  }%
  \providecommand\transparent[1]{%
    \errmessage{(Inkscape) Transparency is used (non-zero) for the text in Inkscape, but the package 'transparent.sty' is not loaded}%
    \renewcommand\transparent[1]{}%
  }%
  \providecommand\rotatebox[2]{#2}%
  \newcommand*\fsize{\dimexpr\f@size pt\relax}%
  \newcommand*\lineheight[1]{\fontsize{\fsize}{#1\fsize}\selectfont}%
  \ifx\svgwidth\undefined%
    \setlength{\unitlength}{102.41859015bp}%
    \ifx\svgscale\undefined%
      \relax%
    \else%
      \setlength{\unitlength}{\unitlength * \real{\svgscale}}%
    \fi%
  \else%
    \setlength{\unitlength}{\svgwidth}%
  \fi%
  \global\let\svgwidth\undefined%
  \global\let\svgscale\undefined%
  \makeatother%
  \begin{picture}(1,0.88204267)%
    \lineheight{1}%
    \setlength\tabcolsep{0pt}%
    \put(0,0){\includegraphics[width=\unitlength,page=1]{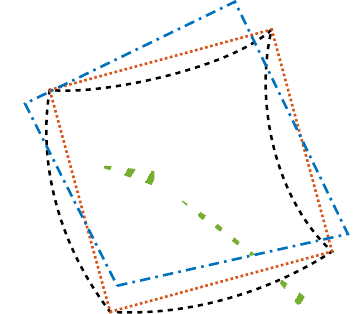}}%
    \put(0.04445317,0.16697963){\color[rgb]{0,0,0}\makebox(0,0)[lt]{\lineheight{1.25}\smash{\begin{tabular}[t]{l}$\preimage_c$\end{tabular}}}}%
    \put(0.71083494,0.8345912){\color[rgb]{0.85098039,0.3254902,0.09803922}\makebox(0,0)[lt]{\lineheight{1.25}\smash{\begin{tabular}[t]{l}$\preimage^\mtext{Rep}_c$\end{tabular}}}}%
    \put(0.1978671,0.73265725){\color[rgb]{0,0.44705882,0.74117647}\makebox(0,0)[lt]{\lineheight{1.25}\smash{\begin{tabular}[t]{l}$\approximate{\preimage_c}$\end{tabular}}}}%
    \put(0,0){\includegraphics[width=\unitlength,page=2]{remapping_errors_rec.pdf}}%
    \put(0.23152441,0.46258608){\color[rgb]{0.46666667,0.67843137,0.18823529}\makebox(0,0)[lt]{\lineheight{1.25}\smash{\begin{tabular}[t]{l}$\approximate{I}^{(n)}$\end{tabular}}}}%
    \put(0.62123628,0.36846523){\color[rgb]{0,0,0}\makebox(0,0)[lt]{\lineheight{1.25}\smash{\begin{tabular}[t]{l}$I^{(n)}$\end{tabular}}}}%
    \put(0,0){\includegraphics[width=\unitlength,page=3]{remapping_errors_rec.pdf}}%
  \end{picture}%
\endgroup%

  }\hfill
  \subcaptionbox{The area of the hatched pattern corresponds to the representation error $E_0^\mtext{Rep}$ ~\cref{eqn:remap:error:rep}.
  \label{fig:remap:remapping_errors_rep}}
  [\threefigwidth]{
    %% Creator: Inkscape inkscape 0.92.5, www.inkscape.org
%% Accompanies image file 'remapping_errors_rep.pdf' (pdf, eps, ps)
\begingroup%
  \makeatletter%
  \providecommand\color[2][]{%
    \errmessage{(Inkscape) Color is used for the text in Inkscape, but the package 'color.sty' is not loaded}%
    \renewcommand\color[2][]{}%
  }%
  \providecommand\transparent[1]{%
    \errmessage{(Inkscape) Transparency is used (non-zero) for the text in Inkscape, but the package 'transparent.sty' is not loaded}%
    \renewcommand\transparent[1]{}%
  }%
  \providecommand\rotatebox[2]{#2}%
  \newcommand*\fsize{\dimexpr\f@size pt\relax}%
  \newcommand*\lineheight[1]{\fontsize{\fsize}{#1\fsize}\selectfont}%
  \ifx\svgwidth\undefined%
    \setlength{\unitlength}{102.41859015bp}%
    \ifx\svgscale\undefined%
      \relax%
    \else%
      \setlength{\unitlength}{\unitlength * \real{\svgscale}}%
    \fi%
  \else%
    \setlength{\unitlength}{\svgwidth}%
  \fi%
  \global\let\svgwidth\undefined%
  \global\let\svgscale\undefined%
  \makeatother%
  \begin{picture}(1,0.88204267)%
    \lineheight{1}%
    \setlength\tabcolsep{0pt}%
    \put(0,0){\includegraphics[width=\unitlength,page=1]{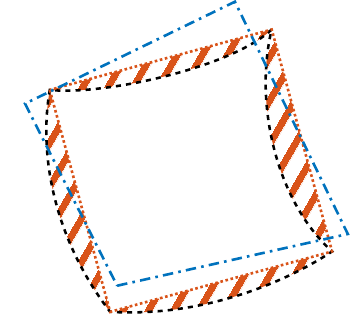}}%
    \put(0.04445317,0.16697963){\color[rgb]{0,0,0}\makebox(0,0)[lt]{\lineheight{1.25}\smash{\begin{tabular}[t]{l}$\preimage_c$\end{tabular}}}}%
    \put(0.71083494,0.8345912){\color[rgb]{0.85098039,0.3254902,0.09803922}\makebox(0,0)[lt]{\lineheight{1.25}\smash{\begin{tabular}[t]{l}$\preimage^\mtext{Rep}_c$\end{tabular}}}}%
    \put(0.1978671,0.73265725){\color[rgb]{0,0.44705882,0.74117647}\makebox(0,0)[lt]{\lineheight{1.25}\smash{\begin{tabular}[t]{l}$\approximate{\preimage_c}$\end{tabular}}}}%
    \put(0,0){\includegraphics[width=\unitlength,page=2]{remapping_errors_rep.pdf}}%
    \put(0.23152441,0.46258608){\color[rgb]{0.46666667,0.67843137,0.18823529}\makebox(0,0)[lt]{\lineheight{1.25}\smash{\begin{tabular}[t]{l}$\approximate{I}^{(n)}$\end{tabular}}}}%
    \put(0.62123628,0.36846523){\color[rgb]{0,0,0}\makebox(0,0)[lt]{\lineheight{1.25}\smash{\begin{tabular}[t]{l}$I^{(n)}$\end{tabular}}}}%
    \put(0,0){\includegraphics[width=\unitlength,page=3]{remapping_errors_rep.pdf}}%
  \end{picture}%
\endgroup%

  }\hfill
  \subcaptionbox{The area of the hatched pattern corresponds to the integration error $E_0^\mtext{Int}$ ~\cref{eqn:remap:error:int}.
  \label{fig:remap:remapping_errors_int}}
  [\threefigwidth]{
    %% Creator: Inkscape inkscape 0.92.5, www.inkscape.org
%% Accompanies image file 'remapping_errors_int.pdf' (pdf, eps, ps)
\begingroup%
  \makeatletter%
  \providecommand\color[2][]{%
    \errmessage{(Inkscape) Color is used for the text in Inkscape, but the package 'color.sty' is not loaded}%
    \renewcommand\color[2][]{}%
  }%
  \providecommand\transparent[1]{%
    \errmessage{(Inkscape) Transparency is used (non-zero) for the text in Inkscape, but the package 'transparent.sty' is not loaded}%
    \renewcommand\transparent[1]{}%
  }%
  \providecommand\rotatebox[2]{#2}%
  \newcommand*\fsize{\dimexpr\f@size pt\relax}%
  \newcommand*\lineheight[1]{\fontsize{\fsize}{#1\fsize}\selectfont}%
  \ifx\svgwidth\undefined%
    \setlength{\unitlength}{102.41859015bp}%
    \ifx\svgscale\undefined%
      \relax%
    \else%
      \setlength{\unitlength}{\unitlength * \real{\svgscale}}%
    \fi%
  \else%
    \setlength{\unitlength}{\svgwidth}%
  \fi%
  \global\let\svgwidth\undefined%
  \global\let\svgscale\undefined%
  \makeatother%
  \begin{picture}(1,0.88204267)%
    \lineheight{1}%
    \setlength\tabcolsep{0pt}%
    \put(0,0){\includegraphics[width=\unitlength,page=1]{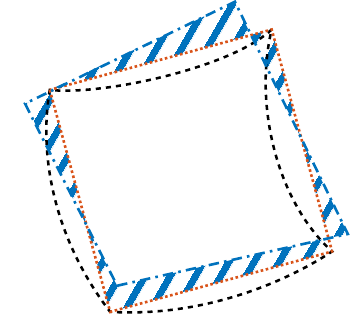}}%
    \put(0.04445317,0.16697963){\color[rgb]{0,0,0}\makebox(0,0)[lt]{\lineheight{1.25}\smash{\begin{tabular}[t]{l}$\preimage_c$\end{tabular}}}}%
    \put(0.71083494,0.8345912){\color[rgb]{0.85098039,0.3254902,0.09803922}\makebox(0,0)[lt]{\lineheight{1.25}\smash{\begin{tabular}[t]{l}$\preimage^\mtext{Rep}_c$\end{tabular}}}}%
    \put(0.1978671,0.73265725){\color[rgb]{0,0.44705882,0.74117647}\makebox(0,0)[lt]{\lineheight{1.25}\smash{\begin{tabular}[t]{l}$\approximate{\preimage_c}$\end{tabular}}}}%
    \put(0,0){\includegraphics[width=\unitlength,page=2]{remapping_errors_int.pdf}}%
    \put(0.23152441,0.46258608){\color[rgb]{0.46666667,0.67843137,0.18823529}\makebox(0,0)[lt]{\lineheight{1.25}\smash{\begin{tabular}[t]{l}$\approximate{I}^{(n)}$\end{tabular}}}}%
    \put(0.62123628,0.36846523){\color[rgb]{0,0,0}\makebox(0,0)[lt]{\lineheight{1.25}\smash{\begin{tabular}[t]{l}$I^{(n)}$\end{tabular}}}}%
    \put(0,0){\includegraphics[width=\unitlength,page=3]{remapping_errors_int.pdf}}%
  \end{picture}%
\endgroup%

  }
  \caption{Illustration of the three error contributions.
  The dashed curvilinear polygon represents the boundary of the preimage of $c$, whereas its polygonal approximation $\preimage^\mtext{Rep}_c$ and fully approximated polygon $\approximate{\preimage_c}$ are denoted by the dotted and dash-dotted lines respectively.
  Furthermore, the interface $I^{(n)}$ at $t = t^{(n)}$ corresponds to the solid curved line, whereas the approximate piecewise linear interface $\approximate{I}^{(n)}$ is represented by the solid lines.}
  \label{fig:remap:remapping_errors}
\end{figure}
Secondly, we approximate the flow map using a numerical integration method.
This means we approximately integrate along pathlines, where the velocity field $\+u(t, \+x)$ is now linearly interpolated (in space and time) from a staggered velocity field.
For the time integration we use the second-order accurate Heun method, which, as we will see, is sufficiently accurate for this purpose.
This results in the following approximate liquid volume
\begin{equation}\label{eqn:remap:approx:zerothmoment}
  \approximate{M_{0,c}^l} = M_0\roundpar{\approximate{P_c} \cap \approximate{\Omega_c^{l}}}.
\end{equation}

The method described here is very similar to the Lagrangian-Eulerian advection scheme (LEAS) presented in~\citet{Zinjala2015}.

\subsection{Error analysis}\label{sec:remap:error}
We will now estimate the error in the volume fractions, due to the two approximations made in the preimage as well as the approximation of the liquid domain.
To this end we denote the symmetric difference of two sets $A, B \subset \mathbb{R}^2$ by
\begin{equation}\label{eqn:remap:symmdiff}
  A \symmdiff B \defeq (A \cup B) \setminus (A \cap B). %(A \cap B^\complement) \cup (A^\complement \cap B) =
\end{equation}
% where $A^\complement$ denotes the complement of $A$ in $\mathbb{R}^2$.
% See also~\cref{fig:remap:error:symmdiff}.
% \begin{figure}
%   \importinkscape{symmetric_difference}
%   \caption{Example of the symmetric difference of a square $A$ and a circle $B$.}
%   \label{fig:remap:error:symmdiff}
% \end{figure}
Using the symmetric difference we define the reconstruction error as
\begin{equation}\label{eqn:remap:error:rec}
  E_0^\mtext{Rec} \defeq M_0\roundpar{\approximate{\Omega_c^{l}} \symmdiff \Omega_c^{l}}.
\end{equation}
Here $\Omega_c^{l}$ denotes the liquid neighbourhood centred around the control volume $c$
\begin{equation}
  \Omega_c^{l} = {\Omega^{l}} \cap \bigcup_{c' \in \mathcal{C}(c)} c',
\end{equation}
where $\mathcal{C}(c)$ is the set of control volumes which share at least one vertex with $c$ (i.e. a $3\times3$ neighbourhood of control volumes if the mesh is rectilinear).
The CFL time step restriction given by~\cref{eqn:intro:cfl} guarantees that the preimage will not overlap with any control volumes other than the ones that share a vertex with $c$, thereby permitting the use of the set $\mathcal{C}(c)$.
Moreover we define the polygonal representation error
\begin{equation}\label{eqn:remap:error:rep}
  E_0^\mtext{Rep} \defeq M_0\roundpar{\preimage^\mtext{Rep}_c \symmdiff \preimage_c},
\end{equation}
and the integration error
\begin{equation}\label{eqn:remap:error:int}
  E_0^\mtext{Int} \defeq M_0\roundpar{\approximate{\preimage_c} \symmdiff \preimage^\mtext{Rep}_c}.
\end{equation}
The three errors are illustrated in~\cref{fig:remap:remapping_errors}.

The following lemma, which is based on~\citet[eqn. 3.15]{Zhang2013}, shows how the error resulting from a Lagrangian remapping method can be bounded by the previously introduced errors.
\begin{lemma}[Error decomposition]\label{lem:error_decomposition}
  The approximation error of a Lagrangian remapping method can be bounded by
\begin{equation}
  \abs{\approximate{M_{0,c}^l} - M_{0,c}^l} \le E_0^\mtext{Rec} + E_0^\mtext{Rep} + E_0^\mtext{Int}.
\end{equation}

\end{lemma}
Furthermore, \cref{lem:error_remapping} provides estimates for the approximation errors of the preimage.
Proofs are found in~\cref{sec:remap_accuracy}.
\begin{lemma}[Remapping error estimates]\label{lem:error_remapping}
  The representation and integration errors are given by
\begin{equation}
  E_0^\mtext{Rep} = \mathcal{O}(h \dt (h + \dt)^2), \quad E_0^\mtext{Int} = \mathcal{O}(h \dt (h + \dt)^2 + h\dt^{q+1}),
\end{equation}
where $q$ is the order of accuracy of the time integration method.

\end{lemma}

Combining both lemmas allows us to prove the following consistency result for the Lagrangian remapping method described in~\cref{sec:remap:method}.
\begin{restatable}[Consistency of a Lagrangian remapping method]{theorem}{theoremfractionbound}\label{thm:intro:fraction_error_bound}
  The following single time step consistency result holds if $\delta \propto h$ (as is the case under the CFL time step restriction given by~\cref{eqn:intro:cfl})
  \begin{equation}
    \norm{\approximate{\volfrac} - \volfrac}_{L^\infty} = \underbrace{\mathcal{O}(h^{r-2})}_\text{reconstruction} + \underbrace{\mathcal{O}(h^2)}_\text{representation} + \underbrace{\mathcal{O}(h^2 + h^q)}_\text{integration}.
  \end{equation}
  where $r$ denotes the local order of accuracy of the liquid domain approximation (e.g. $r = 3$ for the MOF method~\citep{Dyadechko2005})
  \begin{equation}
    E_0^\mtext{Rec} = \mathcal{O}(h^{r}),
  \end{equation}
  and $q$ denotes the order of accuracy of the time integration method (e.g. $q = 2$ for Heun's method).
\end{restatable}
\begin{proof}
  The $L^\infty$-norm of some scalar function $\alpha$, which is defined on the grid, is defined as its maximal value in absolute sense
  \begin{equation}
    \norm{\alpha}_{L^\infty} \defeq \max_{c \in \mathcal{C}} \abs{\alpha_c},
  \end{equation}
  and thus
  \begin{equation}
    \norm{\approximate{\volfrac} - \volfrac}_{L^\infty} = \max_{c \in \mathcal{C}} \frac{\abs{\approximate{M_{0,c}^l} - M_{0,c}^l}}{M_0(c)} \le \frac{E_0^\mtext{Rec} + E_0^\mtext{Rep} + E_0^\mtext{Int}}{h^2}.
  \end{equation}
  which follows from~\cref{lem:error_decomposition}.
  Then, using the result of~\cref{lem:error_remapping} and the assumed local order of accuracy of the liquid domain approximation, we find that (assuming $\delta \propto h$)
  \begin{equation}
    \norm{\approximate{\volfrac} - \volfrac}_{L^\infty} \le \frac{\mathcal{O}(h^{r}) + \mathcal{O}(h^4) + \mathcal{O}(h^4 + h^{q+2})}{h^2} = \mathcal{O}(h^{r-2}) + \mathcal{O}(h^2) + \mathcal{O}(h^2 + h^q).
  \end{equation}
\end{proof}
The result of~\cref{thm:intro:fraction_error_bound} is limited to a single time-step.
When the Lagrangian remapping method is not coupled to a Navier--Stokes solver, then the only possible source of error propagation is via the interface reconstruction.
Hence for generalising~\cref{thm:intro:fraction_error_bound} to a multiple time-step convergence result, as is done in~\citet{Zhang2016}, the stability of the interface reconstruction method w.r.t. perturbations in the reference moments (i.e. errors from the previous time-step) must be analysed.
Thus far we are not aware of any results which prove the accuracy of interface reconstruction methods that are based on cost functions, as we consider here.
For example, in~\citep{Dyadechko2005} it is claimed that the MOF method is third-order accurate, but a proof is not provided.

From~\cref{thm:intro:fraction_error_bound} we find that the accuracy of the volume fractions is limited by the reconstruction accuracy whenever a piecewise linear approximation of the interface is used, for which $r = 3$ (we numerically demonstrate this in~\cref{sec:results:reconstruction} and it is claimed in~\citep{Dyadechko2005} for the MOF method).
An increase of the reconstruction accuracy to $r = 4$, would yield an improvement of the $L^\infty$-error of the volume fractions from first- to second-order accuracy, while using the same definition of the approximate preimage.
This is achieved in~\cref{sec:parabolic} where we consider the piecewise parabolic reconstruction of the interface.
\section{Curvature convergence}\label{sec:curvature}
We now turn to the observed lack of curvature convergence in time-dependent problems, as discussed in~\cref{sec:introduction}.
For simplicity in presentation we assume (only in this section) a uniform rectilinear mesh, and denote each control volume by an index $(i,j)$.
The control volume centroid is now denoted by $\+x_{i,j} = \begin{bmatrix}x_i & y_j \end{bmatrix}^T$.

\subsection{Curvature computation using a local height-function}
\begin{figure}
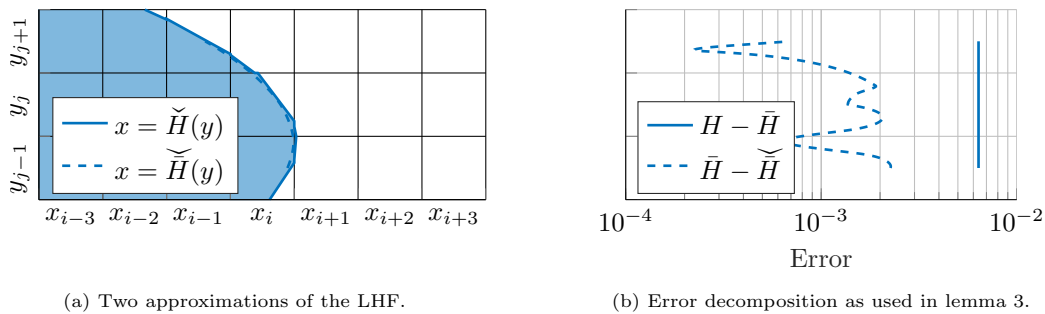

  \setbox9=\hbox{
    \def\tikzWidth{\textwidth*0.35} 
    \def\tikzHeight{\textwidth*0.1714}
    \inputtikzorpdf{curvature_lhf_error}
  }
  \subcaptionbox{Two approximations of the LHF.\label{fig:curvature:lhf:example}}
  [0.5\textwidth]{
    \centering
    \raisebox{\dimexpr\ht9-\height}{
    \def\tikzWidth{\textwidth*0.4}
    \def\tikzHeight{\textwidth*0.1714}
    \inputtikzorpdf{curvature_lhf_example}}
  }\hfill
  \subcaptionbox{Error decomposition as used in~\cref{lem:curvature:error_analysis}.\label{fig:curvature:lhf:error}}
  [0.45\textwidth]{
    \def\tikzWidth{\textwidth*0.35} 
    \def\tikzHeight{\textwidth*0.1714}
    \inputtikzorpdf{curvature_lhf_error} 
  }
  \caption{Example of a part of the fluid interface represented by a LHF for which principal normal direction coincides with the positive $x$-direction.
  The exact interface is a parabola, and therefore the approximation error $H - \bar{H}$, as shown in the right-hand side figure, is constant (solid line).}
  \label{fig:curvature:lhf} 
\end{figure}
The curvature $\kappa$ is computed for each interface control volume using a local height-function (LHF) resulting from a principal normal direction which is aligned with one of the co-ordinate axes\footnote{For unstructured meshes LHFs can also be used, see~\citet{Owkes2015}.}, consider for example~\cref{fig:curvature:lhf:example} where the principal normal direction coincides with the positive $x$-direction.
Given such a LHF, which we denote by $H(y)$, the curvature can be computed as follows
\begin{equation}
  \kappa(y) = H^{[2]}(y)\squarepar{1 + \roundpar{H^{[1]}(y)}^2}^{-3/2},
\end{equation}
where $H^{[1]}, H^{[2]}$ denote the first and second derivative of $H$, respectively.

\def\nrhf{N_H}
Rather than using the LHF directly, the curvature is instead computed using the averaged LHF, which is defined as
\begin{equation}
  \bar{H}(y) \defeq \frac{1}{h}\oneDIntegral{y-h/2}{y+h/2}{H(\tau)}{\tau},
\end{equation}
which by the midpoint rule is a second-order accurate approximation of $H$~\citep{quarteroni2010}.
The reason for using the averaged LHF rather than the LHF itself, is that for rectilinear meshes the former can easily be computed by summing over adjacent volume fractions
\begin{equation}\label{eqn:curvature:average_lhf_sum}
  \bar{H}_j \defeq \bar{H}(y_j) = \frac{1}{h}\sum_{k=-\nrhf}^{\nrhf} h^2 \volfrac_{i+k,j} - (\nrhf + \half)h,
\end{equation}
where we have omitted the dependence on the index $i$ in order to simplify the notation. 
See also~\cref{fig:curvature:lhf:example}.
The parameter $\nrhf$ should be sufficiently large to be able to determine the validity of the averaged LHF: the bottom and top control volumes must be entirely full and empty respectively ($\volfrac_{i-\nrhf,j} = 1, \volfrac_{i+\nrhf,j} = 0$) and the volume fractions must be monotonically decreasing ($\volfrac_{i-k+1,j} \le \volfrac_{i-k,j}$ for $k = -\nrhf, \ldots, \nrhf-1$).
If the principal normal direction coincides with the main direction of the interface normal, then for a sufficiently well-resolved interface these conditions are satisfied for $\nrhf = 3$, as is the case in~\cref{fig:curvature:lhf:example} and used by~\citet{Afkhami2007}.

The averaged LHF can be approximated by summing over adjacent and approximate volume fractions
\begin{equation}
  \approximate{\bar{H}}_j \defeq \frac{1}{h}\sum_{k=-\nrhf}^{\nrhf} h^2 \approximate{\volfrac}_{i+k,j} - (\nrhf + \half)h.% = \bar{H}(y_j) + \mathcal{O}(h^{p+1}),
\end{equation}
Provided that the volume fractions are $p$-th order accurate, this results in a $(p+1)$-st order accurate approximation to $\bar{H}(y)$.
Recall that $p = \min(r-2, 2, q)$ according to~\cref{thm:intro:fraction_error_bound}, where $r$ is the order of accuracy of the liquid domain approximation and $q = 2$ is the order of accuracy of the time integration method.
Hence for traditional piecewise linear interface reconstruction methods, for which $r = 3$, we find that the averaged LHF is approximated at second-order accuracy.

The fully approximated curvature is then given by
\begin{equation}\label{eqn:curvature:definition}
  \approximate{\kappa}_{j} \defeq \approximate{\bar{H}}^{\approximate{[2]}}(y)\squarepar{1 + \roundpar{\approximate{\bar{H}}^{\approximate{[1]}}(y)}^2}^{-3/2},
\end{equation}
where the second-order accurate finite difference approximations to the first and second derivatives of some function $f$ are defined as
\begin{equation}
  f^{\approximate{[1]}}_j \defeq \frac{{f}_{j+1} - {f}_{j-1}}{2h}, \quad f^{\approximate{[2]}}_j \defeq \frac{{f}_{j+1} - 2 {f}_{j} + {f}_{j-1}}{h^2}.
\end{equation}
Hence three consecutive values of the LHF are needed for the LHF based approximation of the interface curvature.
Whenever this is not possible we combine LHFs from different principal normal directions, as proposed in~\citep{Popinet2009}, resulting in the generalised height-function (GHF) method.

\subsection{Error analysis}
The first and second derivatives, which are used to define the approximate curvature in~\cref{eqn:curvature:definition}, are obtained from the approximate differentiation of approximate LHFs.
Hence the error resulting from this approximation consists of an error resulting from the second-order finite difference approximation of the derivative, as well as the approximation of the LHF.
The following lemma provides an estimate of those approximation errors for the first and second derivative of the LHF, and in particular shows that the second derivative converges under mesh refinement only if the volume fractions, which are used to obtain the approximate LHFs, are at least second-order accurate (so $p \ge 2$).
\begin{lemma}[Derivative error estimates]\label{lem:curvature:error_analysis}
  Assuming that $H \in C^\infty([y_j-3h/2,y_j+3h/2])$ and that $p$-th order accurate volume fractions are used to approximate the averaged LHFs, we find
  \begin{equation}\label{eqn:curvature:error_expression}
    \approximate{\bar{H}}^{\approximate{[\lambda]}}_j = H^{[\lambda]}_j + \underbrace{\mathcal{O}(h^2)}_\text{finite difference approx.} + \underbrace{\mathcal{O}(h^{p+1-\lambda})}_\text{volume fraction approx.}, \quad \lambda = 1, 2.
  \end{equation}
\end{lemma}
\begin{proof}
  Let the approximation error of the averaged LHF be defined as
  \begin{equation}\label{eqn:curvature:error_analysis:erbardef}
    \bar{e}(y; h) \defeq \bar{H}(y) - H(y),
  \end{equation}
  where $\bar{e} \in C^\infty([y_j-3h/2,y_j+3h/2])$ by the assumed smoothness of $H$.
  By the midpoint rule we find that $\bar{e}(y; h) = \mathcal{O}(h^2)$.
  It follows that
  \begin{equation}\label{eqn:curvature:error_analysis:ebar}
    \bar{H}^{\approximate{[\lambda]}}_j = H^{\approximate{[\lambda]}}_j + \bar{e}^{\approximate{[\lambda]}}(y_j; h) = \squarepar{H^{{[\lambda]}}_j + \mathcal{O}(h^2)} + \squarepar{\bar{e}^{{[\lambda]}}(y; h) + \mathcal{O}(h^4)} = H^{{[\lambda]}}_j + \mathcal{O}(h^2), \quad \lambda = 1,2,
  \end{equation}
  by applying $(\dots)^{\approximate{[\lambda]}}$ to~\cref{eqn:curvature:error_analysis:erbardef} and evaluating the result at $y = y_j$, and moreover using the differentiability of $\bar{e}$.

  Similarly, we let the approximation error due to the approximation of the volume fractions be defined as
  \begin{equation}
    \approximate{\bar{e}}(y; h) \defeq \approximate{\bar{H}}(y) - \bar{H}(y),
  \end{equation}
  where we have defined
  \begin{equation}
    \approximate{\bar{H}}(y) \defeq \frac{1}{h}\oneDIntegral{y-h/2}{y+h/2}{\approximate{H}(\tau)}{\tau}.
  \end{equation}
  Here $\approximate{H}$ is the approximate LHF based on the piecewise approximation of the interface, see also~\cref{fig:curvature:lhf}.
  By the assumed $p$-th order accuracy of the volume fractions, we find that $\approximate{\bar{e}}(y; h) = \mathcal{O}(h^{p+1})$.
  Contrary to $\bar{e}$, we can not show differentiability of $\approximate{\bar{e}}$ w.r.t. $y$ due to the discontinuity of $\approximate{H}$ at $y = \frac{y_{j-1}+y_{j}}{2}$ and $y = \frac{y_{j}+y_{j+1}}{2}$
   (see also~\cref{fig:curvature:lhf:error} where the error is clearly non differentiable at $y = y_j$).
  This implies that approximation errors in the volume fractions will be amplified due to the division by $2h$ and $h^2$ in the approximation of the first and second derivative respectively
  \begin{equation}\label{eqn:curvature:error_analysis:etildebar}
    \approximate{\bar{H}}^{\approximate{[\lambda]}}_j = \bar{H}^{\approximate{[\lambda]}}_j + \approximate{\bar{e}}^{\approximate{[\lambda]}}(y_j; h) = \bar{H}^{\approximate{[\lambda]}}_j + \mathcal{O}(h^{p+1-\lambda}), \quad \lambda = 1, 2.
  \end{equation}
  Combining~\cref{eqn:curvature:error_analysis:ebar,eqn:curvature:error_analysis:etildebar} yields the desired result.
\end{proof}
We note that the result of~\cref{lem:curvature:error_analysis} is more general than the analysis done in e.g.~\citet{Bornia2011}, where it is assumed that the volume fractions are exact.
That is, letting $p = \infty$ in~\cref{eqn:curvature:error_expression} leads to the same second-order approximation error of the curvature as found in~\citet{Bornia2011}.

\Cref{lem:curvature:error_analysis} implies that when a piecewise linear approximation of the interface is used, for which $r = 3$, we find that the volume fractions will be first order accurate ($p = \min(r-2, 2, q) = 1$) and thus the approximated second derivative is inconsistent due to the insufficiently accurate volume fractions, resulting in a lack of curvature convergence.
If the volume fractions are insufficiently accurate (e.g. $p = 1$) then there will exist some optimal value of $h$, say $h^*$, for which the curvature is most accurate (see e.g.~\citet[eq. 4.20]{Zhang2017}), but of course this precludes convergence under mesh refinement: whenever $h < h^*$ the curvature error will no longer be reduced.
Convergence of the second derivative of the LHF requires a more accurate liquid domain approximation for which $r \ge 4$ (which results in $p = \min(r-2, 2, q) = 2$), in turn this would then also lead to a convergent curvature, as we numerically demonstrate in~\cref{sec:results:reverse,sec:results:translation}.

Obtaining a more accurate liquid domain approximation for which $r = 4$ will be discussed in~\cref{sec:parabolic} where we introduce piecewise parabolic approximations of the interface as generalisations of the piecewise linear approximations of the interface that are discussed in~\cref{sec:plic}.

\section{Optimisation-based PLIC methods}\label{sec:plic}
We will present the piecewise linear interface calculation (PLIC) methods, that are based on the optimisation of some cost function $f$ over some search space $Q$ (e.g. the space of linear interfaces $Q_1$), in a rather abstract way that straightforwardly allows us to generalise any such PLIC method to a parabolic interface reconstruction method.
This generalisation does not alter the cost function $f$, but rather replaces the search space $Q$ with a larger search space which includes parabolic interfaces, which is the topic of~\cref{sec:parabolic}.

We consider two PLIC methods in the context of {optimisation-based reconstruction methods}.
That is, the reconstruction step for each control volume $c$ can be written as an optimisation problem
\begin{equation}\label{eqn:overview:optim}
  q^* = \argmin_{q \in Q_1} f(q),
\end{equation}
where $Q_1$ is the search space consisting of linear interfaces and $f: Q_1 \rightarrow \mathbb{R}$ is a cost function.
More precisely, the search space $Q_1$ is defined as
\begin{equation}\label{eqn:plic:space_p1}
  Q_1 \defeq \set{q(\+x) = \+\eta\cdot (\+x - \+x_c) - \phi(\+\eta; M_{0,c}^l)}{\+\eta \in S^{1}},
\end{equation}
where $S^{1}$ is the unit 1-sphere and $\+x_c$ denotes the centroid of the control volume $c$.
The shift $\phi(\+\eta; M_{0,c}^l)$ is uniquely defined by requiring that the reconstructed liquid volume equals $M_{0,c}^l$, and therefore the only remaining unknown is the interface normal $\+\eta$.

The resulting interface inside the control volume $c$ is defined as the zero level set of $q^* \in Q_1$.
Similarly, the subset of $\mathbb{R}^2$ for which $q \le 0$ (i.e. defining a half-space if $q$ is linear) is denoted by
\begin{equation}
  \halfspace{q} \defeq \set{\+x \in \mathbb{R}^2}{q(\+x) \le 0},
\end{equation}
and is defined such that $c \cap \halfspace{q^*_c}$ defines the subset of $c$ that contains the liquid phase.
Given the optimal level set $q^*_c$ for each control volume $c$, we define the approximate liquid domain as
\begin{equation}
  \approximate{\Omega^l} \defeq \bigcup_{c \in \mathcal{C}} c \cap \halfspace{q^*_c},
\end{equation}
where $\mathcal{C}$ is the set of all control volumes.
See also~\cref{fig:intro:flower_shape_MoF_sd_2}.

Having defined the search space, all that remains is the definition of appropriate cost functions.
Note that such a cost function aims at approximating the actual error in terms of the area of the symmetric difference
\begin{equation}\label{eqn:plic:symmdiff_fun}
  f_{\symmdiff}(q) \defeq \frac{M_0((c \cap \halfspace{q}) \symmdiff (c \cap \Omega^l))}{M_0(c)}.
\end{equation}

\subsection{The least squares VOF interface reconstruction algorithm}
The least squares VOF interface reconstruction algorithm (LVIRA)~\citep{puckett1991volume} uses the following cost function
\begin{equation}\label{eqn:plic:lvira:costfun}
  f_{L^2}(q) \defeq \squarepar{\sum_{c' \in \mathcal{C}(c)} M_0(c') \roundpar{\frac{M^l_{0,c'} - M_0(c' \cap \halfspace{q})}{M_0(c')}}^2}^{1/2},
\end{equation}
where $\mathcal{C}(c)$ denotes the set of control volumes which share at least one vertex with $c$.
That is, the $L^2$-norm of the difference between the actual and reconstructed (by extending the interface to $c'$) volume fractions is considered in a neighbourhood around the control volume $c$.
See also~\cref{fig:plic:lvira:costfun}.
\begin{figure}
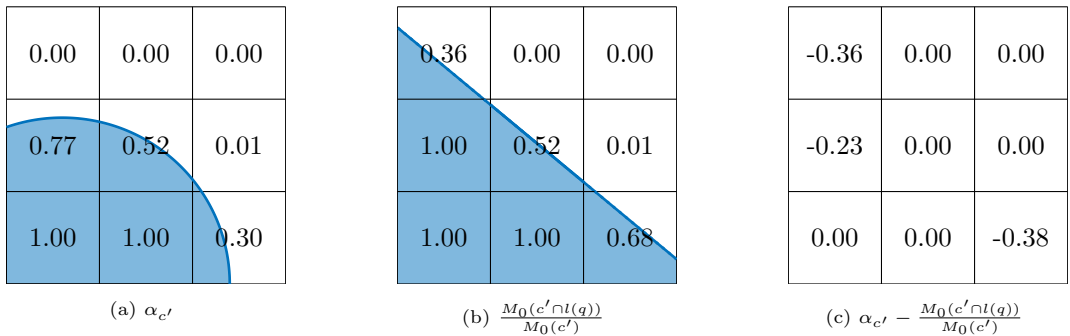

  \subcaptionbox{$\volfrac_{c'}$
  \label{fig:plic:lvira:costfun_exact}}
  [\threefigwidth]{
    \def\tikzWidth{\textwidth*0.25}
    \def\tikzHeight{\textwidth*0.25}
    \inputtikzorpdf{int_reconstruction_L2_costfun_example_exact}
  }\hfill
  \subcaptionbox{$\frac{M_0(c' \cap \halfspace{q})}{M_0(c')}$
  \label{fig:plic:lvira:costfun_reconstruction}}
  [\threefigwidth]{
    \def\tikzWidth{\textwidth*0.25}
    \def\tikzHeight{\textwidth*0.25} 
    \inputtikzorpdf{int_reconstruction_L2_costfun_example_reconstruction}
  }\hfill
  \subcaptionbox{$\volfrac_{c'} - \frac{M_0(c' \cap \halfspace{q})}{M_0(c')}$
  \label{fig:plic:lvira:costfun_difference}}
  [\threefigwidth]{
    \def\tikzWidth{\textwidth*0.25}
    \def\tikzHeight{\textwidth*0.25}
    \inputtikzorpdf{int_reconstruction_L2_costfun_example_difference}
  }
  \caption{Illustration of the cost function $f_{L^2}$ of the LVIRA method (cf.~\cref{eqn:plic:lvira:costfun}): (a) the exact volume fractions, (b) the reconstructed volume fractions (by extension of the interface to the neighbouring control volumes) and (c) their difference.
  Note that due to the computation of the shift $\phi(\+\eta; M_{0,c}^l)$ the central control volume $c$ never contributes to the cost function $f_{L^2}$.}
  \label{fig:plic:lvira:costfun}
\end{figure}
This means that if the volume fractions originate from a globally linear interface, then the reconstruction will be exact.

The cost function $f_{L^2}$ does not always have a unique global minimum\footnote{For example, if for a uniform mesh the $3 \times 3$ volume fraction field $\volfrac_{i,j}$ for $i,j \in \{-1,0,1\}$ has the symmetry $\volfrac_{i,j} = \volfrac_{\pm i,\pm j} \forall i,j$, then multiple optimal solutions exist to~\cref{eqn:overview:optim}.}, but for a resolved interface which can be expressed as a LHF (with any principal normal direction), we find that such problems do not arise.

An efficient LVIRA (ELVIRA) method was proposed~\citep{Pilliod2004} for rectilinear meshes, where six trial normals are considered, and the one for which $f_{L^2}$ is smallest is then selected.
This means that no numerical optimisation is required.
It should be noted however that the resulting normal may not be optimal.
Nevertheless, it can still be shown~\citep{Pilliod2004} that the reconstruction will be exact if the volume fractions originate from a globally linear interface.
% In what follows we will only consider the ELVIRA method, and not the LVIRA method.

\subsection{The moment of fluid reconstruction algorithm}\label{sec:plic:mof}
The moment of fluid (MOF)~\citep{Dyadechko2005} method proposes to include a reference first moment $\+M_{1,c}^{l,*}$ for each control volume $c$ to determine the interface normal.
Given such a reference first moment, the interface normal is determined by requiring the reconstructed first moment to be as close as possible to this reference first moment, this is achieved via optimisation of the following cost function
\begin{equation}\label{eqn:plic:mof:costfun}
  f_{\+M_1}(q) \defeq |\+M_{1,c}^{l,*} - \+M_1(c \cap \halfspace{q})|_2,
\end{equation}
where $\abs{\cdot}_2$ denotes the Euclidean norm.
Contrary to other PLIC methods, the MOF method uses only information (the zeroth and first moment) from the control volume itself, resulting in increased accuracy in particular when the interface is only nearly resolved ($|\kappa| h \approx 1$).
Of course, if the reference moments $M_{0,c}^l$ and $\+M_{1,c}^{l,*}$ originate from a linear interface, then there must exist a normal angle and thus $q^* \in Q_1$ for which $f_{\+M_{1}}(q^*) = 0$ and hence the MOF method reconstructs linear interfaces exactly.
% Note that since the reconstructed volume fraction equals $\volfrac_c$ it follows that this is equivalent to requiring the reconstructed centroid to be as close as possible

For the MOF method it is known that not all reference moments result in stability w.r.t. perturbations of the reference first moment~\citep{Dyadechko2005}, just as there are reference moments which result in multiple solutions\footnote{For example, if the reference centroid coincides with the centroid of a square control volume $c$ while $\volfrac_c < 1$, then for any normal direction $\+\eta \in S^1$ the same value of the cost function $f_{\+M_1}$ is found for any other normal direction $\hat{\+\eta}$ provided that $\abs{\hat\eta_x} = \abs{\eta_x}$ and $\abs{\hat\eta_y} = \abs{\eta_y}$.} to the optimisation~\cref{eqn:overview:optim}.
The set of reference moments for which multiple solutions exist has measure zero, and therefore this does not pose a problem in practice~\citep{Dyadechko2005}.

In two spatial dimensions an efficient variant of the MOF method was proposed~\citep{Lemoine2017} for rectilinear meshes, which uses three trial normals.
The construction of trial normals is such that the trial normal for which $f_{\+M_1}$ is minimal is indeed the optimal one and therefore this efficient variant is merely a new solution strategy for the MOF method.

\subsection{Advection of the first moment}\label{sec:plic:centroid}
The MOF reconstruction method requires a reference first moment, and therefore the first moment must also be advected.
The advection equation of the first moment can be derived by computing the first moment of~\cref{eqn:remap:remap}
\begin{equation}
  \+M_{1,c}^{l,(n+1)} = \+M_1(\Psi^{\delta}(\Psi^{-\delta}c \cap \Omega^{l,(n)})).
\end{equation}
Hence after the preimage of the control volume $c$ is intersected with the liquid domain at $t = t^{(n)}$, this intersection must be advected forward in time before computing the first moment.
We choose to approximate this last step by instead advecting only the centroid of this intersection forward in time (as a point particle), resulting in the following approximation~\citep{Dyadechko2005}
\begin{equation}\label{eqn:plic:moment_advection:point_particle}
  \+C_{c}^{l,(n+1)} = \Psi^{\delta}\roundpar{\+C(\Psi^{-\delta}c \cap \Omega^{l,(n)})} + \mathcal{O}(\delta h^2).
\end{equation}

\subsubsection{Approximate advection of the first moment}
Numerically the advection of the first moment is done in three steps, as illustrated in~\cref{fig:centroid:remap_centroid_advection1,fig:centroid:remap_centroid_advection2,fig:centroid:remap_centroid_advection3}
\begin{figure}
  \subcaptionbox{The centroid of the intersection of the approximate preimage with the approximate liquid domain is computed.
  \label{fig:centroid:remap_centroid_advection1}}
  [\threefigwidth]{
    %% Creator: Inkscape inkscape 0.92.5, www.inkscape.org
%% Accompanies image file 'remap_centroid_advection1.pdf' (pdf, eps, ps)
\begingroup%
  \makeatletter%
  \providecommand\color[2][]{%
    \errmessage{(Inkscape) Color is used for the text in Inkscape, but the package 'color.sty' is not loaded}%
    \renewcommand\color[2][]{}%
  }%
  \providecommand\transparent[1]{%
    \errmessage{(Inkscape) Transparency is used (non-zero) for the text in Inkscape, but the package 'transparent.sty' is not loaded}%
    \renewcommand\transparent[1]{}%
  }%
  \providecommand\rotatebox[2]{#2}%
  \newcommand*\fsize{\dimexpr\f@size pt\relax}%
  \newcommand*\lineheight[1]{\fontsize{\fsize}{#1\fsize}\selectfont}%
  \ifx\svgwidth\undefined%
    \setlength{\unitlength}{119.43349763bp}%
    \ifx\svgscale\undefined%
      \relax%
    \else%
      \setlength{\unitlength}{\unitlength * \real{\svgscale}}%
    \fi%
  \else%
    \setlength{\unitlength}{\svgwidth}%
  \fi%
  \global\let\svgwidth\undefined%
  \global\let\svgscale\undefined%
  \makeatother%
  \begin{picture}(1,0.90553171)%
    \lineheight{1}%
    \setlength\tabcolsep{0pt}%
    \put(0,0){\includegraphics[width=\unitlength,page=1]{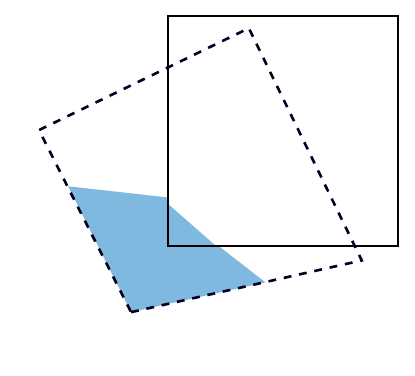}}%
    \put(0.7157611,0.65313317){\color[rgb]{0,0,0}\makebox(0,0)[lt]{\lineheight{1.25}\smash{\begin{tabular}[t]{l}$\approximate{\preimage_c}$\end{tabular}}}}%
    \put(0,0){\includegraphics[width=\unitlength,page=2]{remap_centroid_advection1.pdf}}%
    \put(0.24326032,0.48367702){\color[rgb]{0,0.44705882,0.74509804}\makebox(0,0)[lt]{\lineheight{1.25}\smash{\begin{tabular}[t]{l}$\approximate{I^{(n)}}$\end{tabular}}}}%
    \put(0,0){\includegraphics[width=\unitlength,page=3]{remap_centroid_advection1.pdf}}%
    \put(0.07376424,0.03852951){\color[rgb]{0.11372549,0.11372549,0.11372549}\makebox(0,0)[lt]{\lineheight{1.25}\smash{\begin{tabular}[t]{l}$\approximate{\preimage_c} \cap \approximate{\Omega^{l,(n)}}$\end{tabular}}}}%
    \put(0,0){\includegraphics[width=\unitlength,page=4]{remap_centroid_advection1.pdf}}%
    \put(0.30305585,0.19788502){\color[rgb]{0.11372549,0.11372549,0.11372549}\makebox(0,0)[lt]{\lineheight{1.25}\smash{\begin{tabular}[t]{l}$\approximate{\+C^{l,**}_c}$\end{tabular}}}}%
  \end{picture}%
\endgroup%

  }\hfill
  \subcaptionbox{The centroid is approximately advected as a point particle using~\cref{eqn:plic:moment_advection:point_particle}.
  \label{fig:centroid:remap_centroid_advection2}}
  [\threefigwidth]{
    %% Creator: Inkscape inkscape 0.92.5, www.inkscape.org
%% Accompanies image file 'remap_centroid_advection2.pdf' (pdf, eps, ps)
\begingroup%
  \makeatletter%
  \providecommand\color[2][]{%
    \errmessage{(Inkscape) Color is used for the text in Inkscape, but the package 'color.sty' is not loaded}%
    \renewcommand\color[2][]{}%
  }%
  \providecommand\transparent[1]{%
    \errmessage{(Inkscape) Transparency is used (non-zero) for the text in Inkscape, but the package 'transparent.sty' is not loaded}%
    \renewcommand\transparent[1]{}%
  }%
  \providecommand\rotatebox[2]{#2}%
  \newcommand*\fsize{\dimexpr\f@size pt\relax}%
  \newcommand*\lineheight[1]{\fontsize{\fsize}{#1\fsize}\selectfont}%
  \ifx\svgwidth\undefined%
    \setlength{\unitlength}{119.43349763bp}%
    \ifx\svgscale\undefined%
      \relax%
    \else%
      \setlength{\unitlength}{\unitlength * \real{\svgscale}}%
    \fi%
  \else%
    \setlength{\unitlength}{\svgwidth}%
  \fi%
  \global\let\svgwidth\undefined%
  \global\let\svgscale\undefined%
  \makeatother%
  \begin{picture}(1,0.90553171)%
    \lineheight{1}%
    \setlength\tabcolsep{0pt}%
    \put(0,0){\includegraphics[width=\unitlength,page=1]{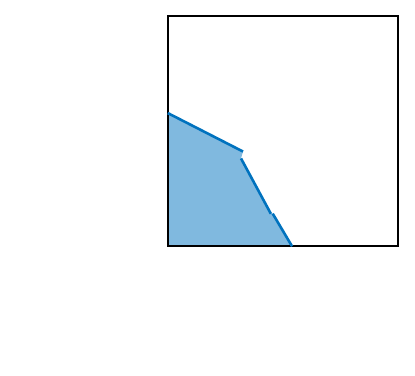}}%
    \put(0.48628688,0.37266592){\color[rgb]{0.11372549,0.11372549,0.11372549}\makebox(0,0)[lt]{\lineheight{1.25}\smash{\begin{tabular}[t]{l}$\approximate{\+C^{l,*}_c}$\end{tabular}}}}%
    \put(0,0){\includegraphics[width=\unitlength,page=2]{remap_centroid_advection2.pdf}}%
    \put(0.30585617,0.3053267){\color[rgb]{0.11372549,0.11372549,0.11372549}\rotatebox{50.617469}{\makebox(0,0)[lt]{\lineheight{1.25}\smash{\begin{tabular}[t]{l}$\dt \+u^{(n)}_c$\end{tabular}}}}}%
  \end{picture}%
\endgroup%

  }\hfill
  \subcaptionbox{The interface is reconstructed by optimising the MOF cost function $f_{\+M_1}$ (\cref{eqn:plic:mof:costfun}).
  \label{fig:centroid:remap_centroid_advection3}}
  [\threefigwidth]{
    %% Creator: Inkscape inkscape 0.92.5, www.inkscape.org
%% Accompanies image file 'remap_centroid_advection3.pdf' (pdf, eps, ps)
\begingroup%
  \makeatletter%
  \providecommand\color[2][]{%
    \errmessage{(Inkscape) Color is used for the text in Inkscape, but the package 'color.sty' is not loaded}%
    \renewcommand\color[2][]{}%
  }%
  \providecommand\transparent[1]{%
    \errmessage{(Inkscape) Transparency is used (non-zero) for the text in Inkscape, but the package 'transparent.sty' is not loaded}%
    \renewcommand\transparent[1]{}%
  }%
  \providecommand\rotatebox[2]{#2}%
  \newcommand*\fsize{\dimexpr\f@size pt\relax}%
  \newcommand*\lineheight[1]{\fontsize{\fsize}{#1\fsize}\selectfont}%
  \ifx\svgwidth\undefined%
    \setlength{\unitlength}{119.43349763bp}%
    \ifx\svgscale\undefined%
      \relax%
    \else%
      \setlength{\unitlength}{\unitlength * \real{\svgscale}}%
    \fi%
  \else%
    \setlength{\unitlength}{\svgwidth}%
  \fi%
  \global\let\svgwidth\undefined%
  \global\let\svgscale\undefined%
  \makeatother%
  \begin{picture}(1,0.90553171)%
    \lineheight{1}%
    \setlength\tabcolsep{0pt}%
    \put(0,0){\includegraphics[width=\unitlength,page=1]{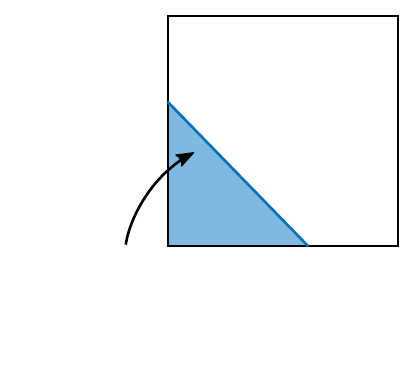}}%
    \put(0.19987922,0.18912055){\color[rgb]{0.11372549,0.11372549,0.11372549}\makebox(0,0)[lt]{\lineheight{1.25}\smash{\begin{tabular}[t]{l}$c \cap \approximate{\Omega^{l,(n+1)}}$\end{tabular}}}}%
    \put(0.51285482,0.58992869){\color[rgb]{0,0.44705882,0.74509804}\makebox(0,0)[lt]{\lineheight{1.25}\smash{\begin{tabular}[t]{l}$\approximate{I^{(n+1)}}$\end{tabular}}}}%
    \put(0,0){\includegraphics[width=\unitlength,page=2]{remap_centroid_advection3.pdf}}%
    \put(0.4553854,0.33674373){\color[rgb]{0.11372549,0.11372549,0.11372549}\makebox(0,0)[lt]{\lineheight{1.25}\smash{\begin{tabular}[t]{l}$\approximate{\+C^{l,(n+1)}_c}$\end{tabular}}}}%
  \end{picture}%
\endgroup%

  }
  \caption{Illustration of the three steps involved in advecting the liquid centroid (first moment) when a Lagrangian remapping method is used.}
  \label{fig:centroid:remap_centroid_advection}
\end{figure}

\begin{step}[Centroid advection]\leavevmode
  \begin{stepenum}
    \item The liquid first moment of the intersection of the approximate preimage with the approximate liquid domain is computed
    \begin{equation}
      \approximate{\+M_{1,c}^{l,**}} = \+M_1\roundpar{\approximate{\preimage_c} \cap \approximate{\Omega_c^{l,(n)}}}.
    \end{equation}
    \label{steps:plic:centroid:intersect}
    \item The centroid of this intersection is advected as a point particle using the linearly interpolated velocity field (in space and time)
    \begin{equation}
      \approximate{\+C_{c}^{l,*}} = \approximate{\+C_{c}^{l,**}} + \delta \+u^{(n)}_c \quad\Rightarrow\quad \approximate{\+M_{1,c}^{l,*}} = \approximate{\+M_{1,c}^{l,**}} + \approximate{M_{0,c}^{l,(n+1)}} \delta \+u^{(n)}_c,
    \end{equation}
    where for simplicity in presentation we consider here the forward Euler method for approximation of the time-integration in~\cref{eqn:plic:moment_advection:point_particle}.
    \label{steps:plic:centroid:advection}
    \item The interface is reconstructed by minimising the MOF cost function (\cref{eqn:plic:mof:costfun}) using $\approximate{\+M_{1,c}^{l,*}}$ as the reference first moment.
    The first moment corresponding to the reconstructed interface, which is defined as the zero level set of some $q^* \in Q$, is denoted by
    \begin{equation}
      \approximate{\+M_{1,c}^{l,(n+1)}} = \+M_1(c \cap \halfspace{q^*}).
    \end{equation}
    \label{steps:plic:centroid:reconstruction}
  \end{stepenum}
\end{step}
Note that for the advection of the zeroth moment, as given by~\cref{eqn:remap:approx:zerothmoment}, we considered only~\cref{steps:plic:centroid:intersect}, since~\cref{steps:plic:centroid:advection,steps:plic:centroid:reconstruction} do not alter the zeroth moment.

\subsubsection{Error analysis}
Each of the steps introduces several errors.
For the error analysis of~\cref{steps:plic:centroid:intersect} we note that for any $A \subseteq \mathbb{R}^2$
\begin{equation}
  \abs{\+M_1(\Omega_c \cap A) - M_0(\Omega_c \cap A)\+x_c}_2 = \abs{\integral{\Omega_c \cap A}{(\+x - \+x_c)}{V}}_2 = \mathcal{O}(h) M_0(\Omega_c \cap A).
\end{equation}
Hence if we define the first moments $\+M_{1,c}^{l,**}, \+M_{1,c}^{l,*}$ and $\+M_{1,c}^{l,(n+1)}$ relative to the centroid of the control volume $\+x_c$, then we can directly use the results of~\cref{sec:remap} to find the estimates
\begin{align}
  \abs{\+E_1^\mtext{Rec}}_2 &\defeq \abs{\+M_1\roundpar{\approximate{\Omega_c^{l}} \symmdiff \Omega_c^{l}}}_2 = \mathcal{O}(h^{r+1})\\
  \abs{\+E_1^\mtext{Rep}}_2 &\defeq \abs{\+M_1\roundpar{\preimage^\mtext{Rep}_c \symmdiff \preimage_c}}_2 = \mathcal{O}(h^2 \delta (h + \delta)^2)\\
  \abs{\+E_1^\mtext{Int}}_2 &\defeq \abs{\+M_1\roundpar{\approximate{\preimage_c} \symmdiff \preimage^\mtext{Rep}_c}}_2 = \mathcal{O}(h^2 \delta (h + \delta)^2 + h^2\delta^{q+1}).
\end{align}
It follows that if $\delta \propto h$
\begin{equation}
  \abs{\approximate{\+M_{1,c}^{l,**}} - \+M_{1,c}^{l,**}}_2 \le \abs{\+M_1\roundpar{\roundpar{\approximate{P_c} \cap \approximate{\Omega_c^{l}}} \symmdiff \roundpar{P_c \cap \Omega_c^{l}}}}_2 = \mathcal{O}(h^{r+1}) + \mathcal{O}(h^5) + \mathcal{O}(h^{q+3}).
\end{equation}
Note that the value of the first moment itself is $\mathcal{O}(h^3)$.

Step~\ref{steps:plic:centroid:advection} introduces the approximation error due to advecting the centroid as a point particle (\cref{eqn:plic:moment_advection:point_particle}) as well as a time integration error, where we again assume that a method of order $q$ is used.
In total this results in the following error estimate for the first two steps
\begin{equation}\label{eqn:plic:mof:steptwoaccuracy}
  \abs{\approximate{\+M_{1,c}^{l,*}} - \+M_{1,c}^{l,*}}_2 = \underbrace{\mathcal{O}(h^{r+1})}_\text{reconstruction} + \underbrace{\mathcal{O}(h^5 + h^{q+3})}_\text{preimage approximation} + \underbrace{\mathcal{O}(h^5)}_\text{\cref{eqn:plic:moment_advection:point_particle}} + \underbrace{\mathcal{O}(h^{q+3})}_\text{time integration}.
\end{equation}

Finally, in~\cref{steps:plic:centroid:reconstruction} the interface is reconstructed according to the MOF cost function (\cref{eqn:plic:mof:costfun}), resulting in some optimal $q^* \in Q$ for which the error made in the third step is minimised
\begin{equation}
  \abs{\approximate{\+M_{1,c}^{l,*}} - \approximate{\+M_{1,c}^{l,(n+1)}}}_2 = \abs{\approximate{\+M_{1,c}^{l,*}} - \+M_1(c \cap \halfspace{q^*})}_2 = f_{\+M_1}(q^*).
\end{equation}
In~\citet{Dyadechko2005} the following estimate is given for the approximation error made in the first moment when the MOF method is used
\begin{equation}\label{eqn:plic:mof:fmaccuracy}
  f_{\+M_1}(q^*) = \mathcal{O}(h^5), \quad q^* \in Q_1,
\end{equation}
which assumes that the interface that is being approximated is a $C^2$ curve with a bounded (from below) radius of curvature.
After a single time-step however, the interface has become piecewise smooth, and for the general non-smooth case the authors of~\citep{Dyadechko2005} provide the pessimistic estimate
\begin{equation}\label{eqn:plic:mof:fmaccuracy_nonsmooth}
  f_{\+M_1}(q^*) = \mathcal{O}(h^3),
\end{equation}
which, since the value of the first moment itself is also $\mathcal{O}(h^3)$, corresponds to a relative error of $\mathcal{O}(1)$.
It follows that, even for a single time-step, we can not prove consistency of the first-moment when the MOF cost function is used.

The numerical results shown in~\cref{fig:int_validation_advection_vortex_T1_LEAS_path_firstMoment} however suggest that
\begin{equation}
  f_{\+M_1}(q^*) = \mathcal{O}(h^4),
\end{equation}
when a PLIC reconstruction is used.
Hence the non-smooth estimate~\cref{eqn:plic:mof:fmaccuracy_nonsmooth} is indeed found to be too pessimistic.

\section{Parabolic reconstruction methods}\label{sec:parabolic}
The PLIC methods introduced in~\cref{sec:plic} result in a local order of accuracy of the liquid domain approximation of $r = 3$ (this is claimed in~\citep{Dyadechko2005} for the MOF method and observed numerically in~\cref{sec:results:reconstruction}).
Therefore, using the result of~\cref{thm:intro:fraction_error_bound} we find that the volume fractions are at most first order accurate, which using~\cref{lem:curvature:error_analysis} with $p = 1$ implies that the second derivative of the LHF, and therefore the curvature, does not converge.
To this end we now consider the generalisation of the reconstruction methods discussed in~\cref{sec:plic} to parabolic reconstruction, which will increase the local order of accuracy of the liquid domain approximation to $r = 4$.

\subsection{Parabolic search spaces}
We define the restricted parabolic search space $Q_2$ as
\begin{equation}\label{eqn:para:space_p2}
  Q_2 \defeq \set{q(\+x) = \+\eta\cdot(\+x - \+x_c) - \phi(\+\eta, \kappa; M_{0,c}^l) + \frac{\kappa}{2}(\+\tau\cdot(\+x - \+x_c))^2}{\+\eta \in S^1, \kappa \in \mathbb{R}},
\end{equation}
where $\+\tau \perp \+\eta$ is the interface tangent and $\kappa$ denotes the curvature.
The search space is restricted in the sense that the parabola can be defined relative to any point $\+x^* \in \mathbb{R}^2$, but we instead fix it around $\+x^* = \+x_c$, this reduces dimensionality of $Q_2$ by one\footnote{The `unrestricted' search space includes a tangential shift $\phi_\tau$ and is given by
\begin{equation}
  \set{q(\+x) = \+\eta\cdot(\+x - \+x_c) - \phi_\eta(\+\eta, \kappa, \phi_\tau; M_{0,c}^l) + \frac{\kappa}{2}(\+\tau\cdot(\+x - \+x_c) - \phi_\tau)^2}{\+\eta \in S^1, \phi_\tau\in \mathbb{R}, \kappa \in \mathbb{R}},
\end{equation}
which corresponds to $\+x^* = \+x_c + \phi_\tau \+\tau$.
}.
Furthermore we denote by $Q_2^\kappa$ the subspace of $Q_2$ where the curvature is known a priori (we let the curvature be approximated by the GHF method \citep{Popinet2009}), this reduces the dimensionality of the search space by one once more.

In~\citet{price2000piecewise} the full parabolic search space was considered: hence solving for the normal angle, curvature as well as the tangential shift $\phi_\tau$.
Here we however limit our discussion to the usage of the parabolic search spaces $Q_2^\kappa$ and $Q_2$.
At this point we have two cost functions, $f_{L^2}$ and $f_{\+M_1}$ (\cref{eqn:plic:lvira:costfun,eqn:plic:mof:costfun}), as discussed in~\cref{sec:plic}, as well as three search spaces: $Q_1, Q^\kappa_2$ and $Q_2$.
In~\cref{tab:para_recon:overview} we show the names of the methods for each of the possible combinations that can be formed.
\begin{table}
  \centering
  \begin{tabular}{rccc}
    \toprule
     & $Q_1$ (\cref{eqn:plic:space_p1}) & $Q^\kappa_2$ & $Q_2$ (\cref{eqn:para:space_p2})\\
     \hline
     $f_{L^2}$ (\cref{eqn:plic:lvira:costfun}) & (E)LVIRA~\citep{puckett1991volume,Pilliod2004} & \textbf{PLVIRA} & PROST~\citep{Renardy2002} \\
     $f_{\+M_1}$ (\cref{eqn:plic:mof:costfun}) & MOF~\citep{Dyadechko2005} & \textbf{PMOF} & N/A\\
    \bottomrule
  \end{tabular}
  \caption{Overview of optimisation-based reconstruction methods under consideration.
  The boldfaced PLVIRA and PMOF methods are the newly proposed methods.
  Here $Q_1$ is the space of linear functions, $Q_2$ denotes the restricted parabolic space and $Q_2^\kappa$ is a subspace of $Q_2$ where the curvature is known beforehand.}
  \label{tab:para_recon:overview}
\end{table}
Reconstruction methods based on $Q^\kappa_2$ or $Q_2$ will be referred to as piecewise parabolic interface calculation (PPIC) methods.

\subsection{PLVIRA and PROST}
The parabolic reconstruction of surface tension (PROST) method~\citep{Renardy2002} uses the $f_{L^2}$ cost function and the $Q_2$ search space in order to accurately compute the interface curvature for use in their numerical surface tension model.
They report the elimination of spurious currents thanks to a balanced surface tension formulation as well as the accurate curvature obtained from the parabolic reconstruction.

We are however quite satisfied with the curvature obtained from the GHF method~\citep{Popinet2009}, and therefore propose to use this curvature in combination with the $Q^\kappa_2$ search space.
This method is referred to as the parabolic LVIRA (PLVIRA) method.
\begin{figure}
  \captionbox{Three cost functions $f_{L^2}, f_{\+M_1}, f_\symmdiff$ for two different search spaces $Q_1, Q_2^\kappa$ as function of the normal angle $\vartheta = \atan(\eta_y/\eta_x)$.
  The used volume fractions and reference first moment correspond to the interface shown in~\cref{fig:int_reconstruction_L2_showRecon,fig:int_reconstruction_M1_showRecon}.
  The markers indicate the optimal angle.
  Note that comparison of the value between two different cost functions does not have any meaning since the cost functions can be arbitrarily scaled while yielding the same optimal angle.
  \label{fig:int_reconstruction_costFun}}
  [\onefigwidth]{
    \def\tikzWidth{\textwidth*0.6}
    \def\tikzHeight{\textwidth*0.375}
    \inputtikzorpdf{int_reconstruction_costFun}
    \begin{minipage}[b][5cm][t]{0.25\textwidth}
      \tikzexternaldisable
      \definecolor{mycolor1}{rgb}{0.00000,0.44700,0.74100}%
\definecolor{mycolor2}{rgb}{0.85000,0.32500,0.09800}%
\definecolor{mycolor3}{rgb}{0.46600,0.67400,0.18800}%

% NB https://tex.stackexchange.com/questions/348229/tabular-like-legend-in-pgfplots
\begin{tikzpicture}
    \begin{axis}[
        hide axis,
        width=2.cm,
        height=2cm,
        ymin=0.99,
        ymax=1,
    ]
        \addplot [line width=1.0pt, color=mycolor1, domain=-0.1:0.1,samples=2, draw=none] {1};
            \label{plot:recon_costfun:line1}
        \addplot [dashed, line width=1.0pt, color=mycolor1, domain=-0.1:0.1,samples=2, draw=none] {1};
            \label{plot:recon_costfun:line2}
        \addplot [line width=1.0pt, color=mycolor2, domain=-0.1:0.1,samples=2, draw=none] {1};
            \label{plot:recon_costfun:line3}
        \addplot [dashed, line width=1.0pt, color=mycolor2, domain=-0.1:0.1,samples=2, draw=none] {1};
            \label{plot:recon_costfun:line4}
        \addplot [line width=1.0pt, color=mycolor3, domain=-0.1:0.1,samples=2, draw=none] {1};
            \label{plot:recon_costfun:line5}
        \addplot [dashed, line width=1.0pt, color=mycolor3, domain=-0.1:0.1,samples=2, draw=none] {1};
            \label{plot:recon_costfun:line6}

        % create a (dummy) coordinate where we want to place the legend
        %
        % (The matrix cannot be placed inside the `axis' environment
        %  directly, because then a catcode error is raised.
        %  I guess that this is caused by the `matrix of nodes' feature)
        \coordinate (legend) at (axis description cs:0.5,0.5);
    \end{axis}

    % create the legend matrix which is placed at the created (dummy) coordinate
    % and recall the plot specification using the `\ref' command
    %
    % adapt the style of that node to your needs
    % (e.g. if you like different spacings between the rows or columns)
    \draw (0,0) node{
        \begin{tabular}{|lcc|}
            % \toprule
            \hline
                                & $Q_1$                             & $Q^\kappa_2$\\
            \hline
            $f_{L^2}$           & \ref{plot:recon_costfun:line1}   & \ref{plot:recon_costfun:line2}\\
            $f_{\mathbf{M}_1}$  & \ref{plot:recon_costfun:line3}   & \ref{plot:recon_costfun:line4}\\
            $f_{\symmdiff}$     & \ref{plot:recon_costfun:line5}   & \ref{plot:recon_costfun:line6}\\
            % \bottomrule
            \hline
        \end{tabular}
    };

\end{tikzpicture}
      \tikzexternalenable
    \end{minipage}
  }

  \captionbox{Exact interface (solid circular arc) as well as reconstructed interfaces using the LVIRA (solid line) and PLVIRA (dashed curve) method.
  \label{fig:int_reconstruction_L2_showRecon}}
  [\twofigwidth]{
    \def\tikzWidth{\textwidth*0.24}
    \def\tikzHeight{\tikzWidth}
    \inputtikzorpdf{int_reconstruction_L2_showRecon}
  }\hfill
  \captionbox{Exact interface (solid circular arc) as well as reconstructed interfaces using the MOF (solid line) and PMOF (dashed curve) method.
  The circle marker corresponds to the reference centroid and the other markers (which overlap) correspond to the reconstructed centroids.
  \label{fig:int_reconstruction_M1_showRecon}}
  [\twofigwidth]{
    \def\tikzWidth{\textwidth*0.24}
    \def\tikzHeight{\tikzWidth}
    \inputtikzorpdf{int_reconstruction_M1_showRecon}
  }
\end{figure}
In~\cref{fig:int_reconstruction_costFun} we show the cost function $f_{L^2}$ for both search spaces as well as the optimal angles and also the error measured in terms of the symmetric difference~\cref{eqn:plic:symmdiff_fun}.
We note that $f_{L^2}(Q_2^\kappa)$ approximates $f_{\symmdiff}(Q_2^\kappa)$ quite well.
In~\cref{fig:int_reconstruction_L2_showRecon} we show the resulting reconstructed interfaces for each of the optimal angles.
As expected, the parabolic interface yields a significant increase in approximation accuracy of the interface.

\subsection{PMOF}
Parabolic reconstruction using the MOF cost function $f_{\+M_1}$ is less straightforward.
The reason for this is that the amount of information that is available is rather small: we have two degrees of freedom ($\+\eta \in S^1, \kappa \in \mathbb{R}$) and only the first moment $\+M^{l,*}_1 \in \mathbb{R}^2$ as reference data (recall that the volume fraction uniquely defines the shift $\phi(\+\eta, \kappa; M_{0,c}^l)$).
The following example illustrates that the MOF cost function cannot be used in combination with the search space $Q_2$.
\begin{example}
  Consider an interface defined by the level set $q(\+x) = y - \phi + \frac{\kappa}{2} x^2$ in the unit control volume $c = [-\frac{1}{2}, \frac{1}{2}]^2$, where $|\phi - \frac{\kappa}{8}| < \frac{1}{2}$.
  The liquid volume fraction can then be computed as
  \begin{equation}
    M_{0,c}^l = \oneDIntegral{-\frac{1}{2}}{\frac{1}{2}}{\oneDIntegral{-\frac{1}{2}}{\phi - \frac{\kappa}{2} x^2}{}{y}}{x} = \frac{1}{2} + \phi  - \frac{\kappa}{24}.
  \end{equation}
  Similarly we compute the first moment
  \begin{equation}
    \+M_1(c^l) = \oneDIntegral{-\frac{1}{2}}{\frac{1}{2}}{\oneDIntegral{-\frac{1}{2}}{\phi - \frac{\kappa}{2} x^2}{\+x}{y}}{x} = \frac{1}{2}\begin{bmatrix}
      0\\
      \frac{\kappa^2}{720} + M_{0,c}^l(M_{0,c}^l-1)
    \end{bmatrix},
  \end{equation}
  where we substituted $\phi = M_{0,c}^l - \frac{1}{2} + \frac{\kappa}{24}$.
  It follows that for this example the sign of the curvature does not affect the first moment, resulting in at least two distinct solutions to the optimisation problem~\cref{eqn:overview:optim}.
\end{example}
Hence for using the MOF method with a parabolic reconstruction we must either decrease the dimensionality of the search space, or add more information such as higher order moments or some information from neighbouring control volumes.

We choose for the former option: the parabolic MOF (PMOF) method refers to the combination of $f_{\+M_1}$ and $Q_2^\kappa$.
Note that we could also interpret this as the latter option, since we include information from neighbouring control volumes indirectly via the GHF curvature.

In~\cref{fig:int_reconstruction_costFun} we show the cost function $f_{\+M_1}$ for both search spaces as well as the optimal angles and also the error measured in terms of the symmetric difference.
The resulting reconstructed interfaces are shown in~\cref{fig:int_reconstruction_M1_showRecon}.
We find, for this example, that even though the PMOF method yields similar accuracy in terms of the cost function $f_{\+M_1}$ when compared to the MOF method, the interface approximation is significantly improved to the point where it can hardly be visually distinguished from the exact interface, which is a circular arc.

\subsection{An algorithm for intersecting a polygon with a parabola}\label{sec:parabolic:intersect}
The evaluation of both cost functions, as well as the computation of the right-hand side of~\cref{eqn:remap:remap_m0} require the computation of either the zeroth or first moment of the intersection of some polygon $\preimage$ with a parabola\footnote{The PROST method~\citep{Renardy2002} computes this intersection approximately.
However since we consider 2D only, the intersection can still rather easily be computed exactly.} defined as the zero level set of some $q \in Q_2$.
That is, we want to compute
\begin{equation}
  M_0(\preimage \cap \halfspace{q}), \quad \+M_1(\preimage \cap \halfspace{q}).
\end{equation}
The description of the algorithm is limited to to the case $\kappa \le 0$.
This is not restrictive, since if $q$ corresponds to an interface with positive curvature then we instead compute the moments using
\begin{equation}
  M_0(\preimage \cap \halfspace{q}) = M_0(\preimage) - M_0(\preimage \cap (\mathbb{R}^2 \setminus \halfspace{q})) = M_0(\preimage) - M_0(\preimage \cap \halfspace{-q}),
\end{equation}
where $-q$ again corresponds to an interface of negative curvature.

\begin{figure}\subcaptionbox{The original polygon $\preimage$ and parabola defined by the level set function $q$.
  \label{fig:app:parabola_polygon:example_1}}
  [\threefigwidth]{
    \import{inkscape/}{parabola_polygon_intersection_1\grayscale.pdf_tex}
  }\hfill
  \subcaptionbox{The new polygon $\hat{\preimage}$ (consisting of two parts) as well as the parabola.
  \label{fig:app:parabola_polygon:example_2}}
  [\threefigwidth]{
    %% Creator: Inkscape inkscape 0.92.5, www.inkscape.org
%% Accompanies image file 'parabola_polygon_intersection_2.pdf' (pdf, eps, ps)
\begingroup%
  \makeatletter%
  \providecommand\color[2][]{%
    \errmessage{(Inkscape) Color is used for the text in Inkscape, but the package 'color.sty' is not loaded}%
    \renewcommand\color[2][]{}%
  }%
  \providecommand\transparent[1]{%
    \errmessage{(Inkscape) Transparency is used (non-zero) for the text in Inkscape, but the package 'transparent.sty' is not loaded}%
    \renewcommand\transparent[1]{}%
  }%
  \providecommand\rotatebox[2]{#2}%
  \newcommand*\fsize{\dimexpr\f@size pt\relax}%
  \newcommand*\lineheight[1]{\fontsize{\fsize}{#1\fsize}\selectfont}%
  \ifx\svgwidth\undefined%
    \setlength{\unitlength}{153.58048931bp}%
    \ifx\svgscale\undefined%
      \relax%
    \else%
      \setlength{\unitlength}{\unitlength * \real{\svgscale}}%
    \fi%
  \else%
    \setlength{\unitlength}{\svgwidth}%
  \fi%
  \global\let\svgwidth\undefined%
  \global\let\svgscale\undefined%
  \makeatother%
  \begin{picture}(1,0.70247997)%
    \lineheight{1}%
    \setlength\tabcolsep{0pt}%
    \put(0,0){\includegraphics[width=\unitlength,page=1]{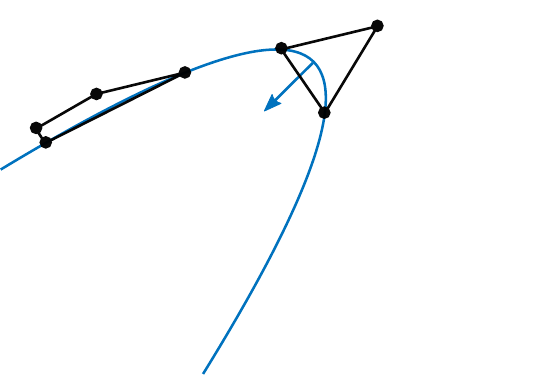}}%
    \put(0.59745562,0.66290382){\color[rgb]{0,0,0}\makebox(0,0)[lt]{\lineheight{1.25}\smash{\begin{tabular}[t]{l}$\hat{e}_1$\end{tabular}}}}%
    \put(0.52342414,0.48861511){\color[rgb]{0,0,0}\makebox(0,0)[lt]{\lineheight{1.25}\smash{\begin{tabular}[t]{l}$\hat{e}_2$\end{tabular}}}}%
    \put(0.6776171,0.53957792){\color[rgb]{0,0,0}\makebox(0,0)[lt]{\lineheight{1.25}\smash{\begin{tabular}[t]{l}$\hat{e}_3$\end{tabular}}}}%
    \put(0.2361125,0.5728759){\color[rgb]{0,0,0}\makebox(0,0)[lt]{\lineheight{1.25}\smash{\begin{tabular}[t]{l}$\hat{e}_4$\end{tabular}}}}%
    \put(0.08442241,0.52724537){\color[rgb]{0,0,0}\makebox(0,0)[lt]{\lineheight{1.25}\smash{\begin{tabular}[t]{l}$\hat{e}_5$\end{tabular}}}}%
    \put(-0.00067208,0.40638686){\color[rgb]{0,0,0}\makebox(0,0)[lt]{\lineheight{1.25}\smash{\begin{tabular}[t]{l}$\hat{e}_6$\end{tabular}}}}%
    \put(0.22131349,0.44461777){\color[rgb]{0,0,0}\makebox(0,0)[lt]{\lineheight{1.25}\smash{\begin{tabular}[t]{l}$\hat{e}_7$\end{tabular}}}}%
  \end{picture}%
\endgroup%

  }\hfill
  \subcaptionbox{Illustration of the correction resulting from $\hat{e}_2 \in \hat{\mathcal{E}}^*$ (hatched region).
  \label{fig:app:parabola_polygon:example_3}}
  [\threefigwidth]{
    %% Creator: Inkscape inkscape 0.92.5, www.inkscape.org
%% Accompanies image file 'parabola_polygon_intersection_3.pdf' (pdf, eps, ps)
\begingroup%
  \makeatletter%
  \providecommand\color[2][]{%
    \errmessage{(Inkscape) Color is used for the text in Inkscape, but the package 'color.sty' is not loaded}%
    \renewcommand\color[2][]{}%
  }%
  \providecommand\transparent[1]{%
    \errmessage{(Inkscape) Transparency is used (non-zero) for the text in Inkscape, but the package 'transparent.sty' is not loaded}%
    \renewcommand\transparent[1]{}%
  }%
  \providecommand\rotatebox[2]{#2}%
  \newcommand*\fsize{\dimexpr\f@size pt\relax}%
  \newcommand*\lineheight[1]{\fontsize{\fsize}{#1\fsize}\selectfont}%
  \ifx\svgwidth\undefined%
    \setlength{\unitlength}{141.32186096bp}%
    \ifx\svgscale\undefined%
      \relax%
    \else%
      \setlength{\unitlength}{\unitlength * \real{\svgscale}}%
    \fi%
  \else%
    \setlength{\unitlength}{\svgwidth}%
  \fi%
  \global\let\svgwidth\undefined%
  \global\let\svgscale\undefined%
  \makeatother%
  \begin{picture}(1,0.70707839)%
    \lineheight{1}%
    \setlength\tabcolsep{0pt}%
    \put(0,0){\includegraphics[width=\unitlength,page=1]{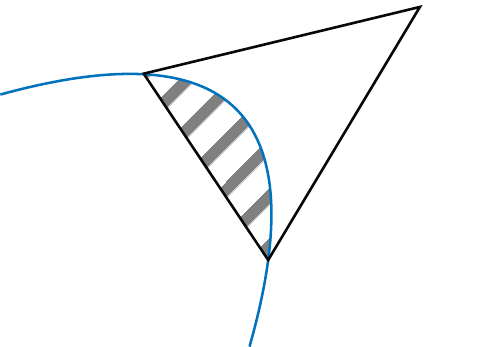}}%
    \put(0.57426065,0.08722577){\color[rgb]{0,0,0}\makebox(0,0)[lt]{\lineheight{1.25}\smash{\begin{tabular}[t]{l}$\hat{\+v}_l$\end{tabular}}}}%
    \put(0.18062093,0.63099942){\color[rgb]{0,0,0}\makebox(0,0)[lt]{\lineheight{1.25}\smash{\begin{tabular}[t]{l}$\hat{\+v}_r$\end{tabular}}}}%
    \put(0.41280316,0.54570789){\color[rgb]{0,0.44705882,0.74509804}\makebox(0,0)[lt]{\lineheight{1.25}\smash{\begin{tabular}[t]{l}$\eta = s - \frac{\kappa}{2} \tau^2$\end{tabular}}}}%
    \put(0.01412609,0.57992242){\color[rgb]{0,0,0.13333333}\rotatebox{-55.50414067}{\makebox(0,0)[lt]{\lineheight{1.25}\smash{\begin{tabular}[t]{l}$\eta = \frac{\hat\eta_{r}(\tau - \hat\tau_{l}) + \hat\eta_l(\hat\tau_{r} - \tau)}{\hat\tau_{r} - \hat\tau_{l}}$\end{tabular}}}}}%
    \put(0,0){\includegraphics[width=\unitlength,page=2]{parabola_polygon_intersection_3.pdf}}%
  \end{picture}%
\endgroup%

  }
  \caption{Intersection of a polygon $\preimage$ with a parabola $q \in Q_2$.}
  \label{fig:app:parabola_polygon:example}
\end{figure}
Let the vertices of the polygon $\preimage$ be denoted by $\+v_k$, for $k = 1, \ldots, N_\preimage$.
Using the level set function $q$ we can determine for each vertex whether it is below ($q(\+v_k) < 0$) or above ($q(\+v_k) > 0$) the parabola, and therefore decide whether the vertex is retained or removed respectively.
For each edge we can then determine to which of the following cases the edge belongs (see also~\cref{fig:app:parabola_polygon:example_1}).
\begin{enumerate}
  \item Both vertices are removed, since $\kappa \le 0$ it follows that the edge is not trisected and is simply removed (see edge $e_4$)
  \item One of the vertices is retained, hence the edge is bisected (see edges $e_3$ and $e_5$)
  \item Both vertices are retained
  \begin{enumerate}
    \item There is no intersection of the edge with the parabola, the edge is retained (see edge $e_2$)
    \item There is an intersection of the edge with the parabola, the edge is trisected and the section in the middle (which may have zero length) is removed (see edge $e_1$)
  \end{enumerate}
\end{enumerate}
For determining whether an intersection exists we parametrize the edge $e_k$ by
\begin{equation}
  e_k = \set{\+e_k(t) = t \+v_k + (1-t) \+v_{k+1}}{t \in [0, 1]},
\end{equation}
and solve the quadratic equation $q(\+e_k(t)) = 0$ for $t$.
If no real root exists in the interval $[0, 1]$ then the edge is not trisected.
We use the same approach for determining the new bi or trisection vertices of edges.

We then form a new polygon $\hat{\preimage}$ (with edges $\hat{\mathcal{E}}$, see~\cref{fig:app:parabola_polygon:example_2}) which is defined by the insertion and removal of vertices and edges according to the previous description.
A subset of the edges is not actually part of the intersection and should in fact be `replaced' by part of the parabola, we denote these edges by $\hat{\mathcal{E}}^*$ (in the example we find $\hat{\mathcal{E}}^* = \{\hat{e}_2, \hat{e}_7\}$).

The intersection volume $M_0(\preimage \cap \halfspace{q})$ can now be computed as the zeroth moment of $\hat{\preimage}$ with corrections that stem from the edges $\hat{\mathcal{E}}^*$, shown as the hatched region in~\cref{fig:app:parabola_polygon:example_3} for the edge $\hat{e}_2$.
For the computation of such a correction we parametrise each such edge by the normal and tangential co-ordinates relative to the control volume centroid $\+x_c$.
That is, if $\hat{\+v}_l$ and $\hat{\+v}_r$ denote the endpoints of such an edge, then the tangential and normal co-ordinates of the vertex $\hat{\+v}_l$ are given by
\begin{equation}
  \hat\tau_l = \+\tau \cdot (\hat{\+v}_l - \+x_c), \quad \hat\eta_l = \+\eta \cdot (\hat{\+v}_l - \+x_c),
\end{equation}
respectively, and similarly for $\hat{\+v}_r$.
Provided with these co-ordinates we can parametrise an edge $\hat{e}$ as follows
\begin{equation}
  \hat{e} = \set{\tau \+\tau + \frac{\hat\eta_{r}(\tau - \hat\tau_{l}) + \hat\eta_l(\hat\tau_{r} - \tau)}{\hat\tau_{r} - \hat\tau_{l}}\+\eta}{\tau \in [\hat\tau_l, \hat\tau_r]}.
\end{equation}
Similarly, the parabola can be expressed as
\begin{equation}
  \eta = \phi - \frac{\kappa}{2} \tau^2.
\end{equation}
It follows that the corrected volume can be expressed as an integral between the parabola and the edge $\hat{e}$, resulting in
\begin{equation}
  M_0(\preimage \cap \halfspace{q}) = M_0(\hat{\preimage}) + \sum_{\hat{e} \in \hat{\mathcal{E}}^*} \oneDIntegral{\hat\tau_l}{\hat\tau_{r}}{\oneDIntegral{\frac{\hat\eta_{r}(\tau - \hat\tau_{l}) + \hat\eta_l(\hat\tau_{r} - \tau)}{\hat\tau_{r} - \hat\tau_{l}}}{\phi - \frac{\kappa}{2}\tau^2}{}{\eta}}{\tau},
\end{equation}
see also~\cref{fig:app:parabola_polygon:example_3}.
The first moment can be computed in a similar way
\begin{equation}
  \+M_1(\preimage \cap \halfspace{q}) = \+M_1(\hat{\preimage}) + \begin{bmatrix}
    \+\eta & \+\tau
  \end{bmatrix}
  \squarepar{\sum_{\hat{e} \in \hat{\mathcal{E}}^*} \oneDIntegral{\hat\tau_l}{\hat\tau_{r}}{\oneDIntegral{\frac{\hat\eta_{r}(\tau - \hat\tau_{l}) + \hat\eta_l(\hat\tau_{r} - \tau)}{\hat\tau_{r} - \hat\tau_{l}}}{\phi - \frac{\kappa}{2}\tau^2}{\begin{bmatrix}
    \eta \\ \tau
  \end{bmatrix}}{\eta}}{\tau}}.
\end{equation}
Note that the correction terms for the volume as well as the first moment are integrals of polynomial functions, and can therefore be evaluated analytically.
An implementation of the intersection algorithm can be found in~\citet{intersection_alg}.

\subsection{Volume enforcement}
For PLIC methods the computation of $\phi(\+\eta; M_{0,c}^l)$ can be done directly using the methods described in~\citet{Scardovelli2000}.
In principle we could distinguish all possible interface configurations for PPIC methods as well and derive analytical formulas for $\phi(\+\eta, \kappa; M_{0,c}^l)$.
For the time being, we instead define $\phi(\+\eta, \kappa; M_{0,c}^l)$ as the root of the following function
\begin{equation}\label{eqn:para:enforcement:g}
  g(\phi) = M_0\squarepar{c \cap \halfspace{\+\eta\cdot(\+x - \+x_c) - \phi + \frac{\kappa}{2}(\+\tau\cdot(\+x - \+x_c))^2}} - M_{0,c}^l,
\end{equation}
which ensures that the liquid volume resulting from the reconstructed interface corresponds to the reference liquid volume $M_{0,c}^l$.
We use Brent's algorithm~\citep{Brent1971} to find the root of~\cref{eqn:para:enforcement:g}, which is guaranteed to converge provided an initial interval $[\phi_-, \phi_+]$ is given for which $g(\phi_-) g(\phi_+) < 0$.
In~\cref{sec:app:brent_bound} we explain how the initial interval $[\phi_-, \phi_+]$ is determined. 
Using these initial intervals we require six evaluations of $g$ on average to obtain machine accuracy in the relative error.

\subsection{Numerical optimisation}
We use the limited-memory Broyden Fletcher Goldfarb Shannon (LBFGS) algorithm \citep{liu1989limited} in conjunction with the More-Thuente line search algorithm \citep{more1994line} for solving the optimisation problem given by~\cref{eqn:overview:optim}.
For the LVIRA cost function we apply the algorithm to a nondimensionalised variant of $f_{L^2}^2$, whereas for the MOF cost function we apply the algorithm to a nondimensionalised variant of $f_{\+M_1}$.
The optimisation requires the computation of derivatives of the cost function (w.r.t. the normal angle $\vartheta$ as well as the curvature $\kappa$ if the PROST method is used), and to this end we have derived the exact linearisation of the cost functions in~\cref{sec:app:linearisations}.

The optimisation algorithm is terminated if the error estimate is below $\min(10^{-2}, (h / L)^2)$, where $L$ is a typical length scale in the problem under consideration.
The tolerances introduced here are determined via numerical experiments, with the goal of ensuring that the tolerance is sufficiently small such that the resulting reconstruction accuracy is no longer affected by the accuracy to which the optimisation problem~\cref{eqn:overview:optim} is solved.

\section{Results}\label{sec:results}
Here we investigate the accuracy of the proposed reconstruction methods, as well as their coupling to the Lagrangian remapping method and finally also fully coupled (via surface tension) to the ComFLOW~\citep{Kleefsman2005,Wemmenhove,VanderPlas2017} two-phase Navier--Stokes solver.

\subsection{Reconstruction}\label{sec:results:reconstruction}
We consider the reconstruction accuracy using a `flower shape' defined by the zero level set of
\begin{equation}\label{eqn:results:recon:levelset}
  q(\+x) = R \curlypar{1 + 0.1 \sin\squarepar{0.1 + 5 \atan\roundpar{\frac{y - y_0}{x - x_0}}}} - \abs{\+x - \+x_0}_2,
\end{equation}
centred around $\+x_0 = \begin{bmatrix}0.01 & 0.03\end{bmatrix}^T$ where the computational domain is given by $\Omega = [-0.5, 0.5]^2$.
We let $R = 0.3$.
An example is shown in~\cref{fig:intro:flower_shape}.
The zeroth and first moments are initialised exactly and the resulting reconstruction accuracy is measured using the area of the symmetric difference between the reconstructed liquid domain and the exact liquid domain as well as the $L^\infty$-error of the first moment.
\begin{figure}
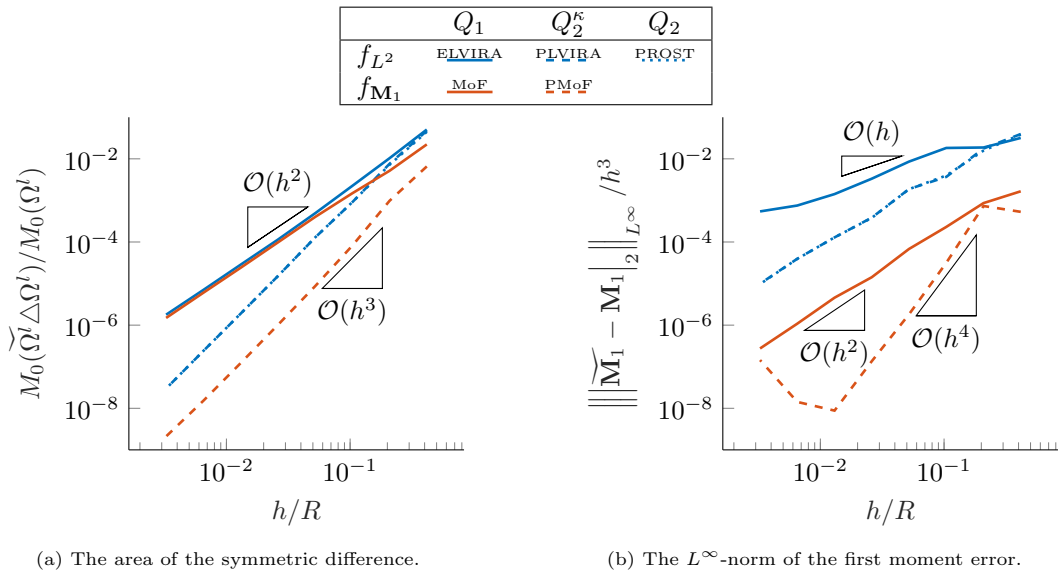

  \begin{subfigure}[t]{\textwidth}
    \centering
    \inputtikzorpdf{int_validation_optim_legend2}
  \end{subfigure}
  
  \subcaptionbox{The area of the symmetric difference.
  \label{fig:int_validation_recon_cmf_symmDiff_flower}}
  [\twofigwidth]{
    \def\tikzWidth{\textwidth*0.3}
    \def\tikzHeight{\textwidth*0.3}
    \inputtikzorpdf{int_validation_recon_cmf_symmDiff_flower}
  }\hfill
  \subcaptionbox{The $L^\infty$-norm of the first moment error.
  \label{fig:int_validation_recon_cmf_firstMoment_flower}}
  [\twofigwidth]{
    \def\tikzWidth{\textwidth*0.3}
    \def\tikzHeight{\textwidth*0.3}
    \inputtikzorpdf{int_validation_recon_cmf_firstMoment_flower}
  }
  \caption{Reconstruction accuracy measured in terms of the symmetric difference as well as the $L^\infty$-norm of the first moment (normalised by $h^3$).
    Exact reference moments were used corresponding to the flower shape~\cref{eqn:results:recon:levelset}.
    Note that the results for PLVIRA and PROST essentially overlap.
  }
  \label{fig:int_validation_reconstruction}
\end{figure}

In~\cref{fig:int_validation_recon_cmf_symmDiff_flower} we show the accuracy in terms of the area of the symmetric difference.
No significant difference in reconstruction accuracy is found between the PROST and PLVIRA methods, confirming that indeed the GHF curvature is sufficiently accurate.
We obtain one additional order of accuracy when using PPIC as compared to PLIC, as desired.
A significant increase in accuracy is found for the PMOF method as compared to the PLVIRA method (about one order of magnitude).

These results seem to suggest that the interface representation accuracy goes up by one order when using $Q_2^\kappa$ or $Q_2$ as compared to using $Q_1$.
Hence in the context of~\cref{thm:intro:fraction_error_bound}, we find that $r = 4$ and we are therefore hopeful that the curvature may now also converge for time-dependent problems.

Furthermore, in~\cref{fig:int_validation_recon_cmf_firstMoment_flower} we show the maximum error in approximating the first moment, which coincides with the maximal value of the MOF cost function~\cref{eqn:plic:mof:costfun}
\begin{equation}
  \norm{\approximate{\+M_1} - \+M_1}_{L^\infty} = \max_{c \in \mathcal{C}} f_{\+M_1}(q^*_c).
\end{equation}
The error is shown relative to the magnitude of the first moment ($\mathcal{O}(h^3)$).
As expected, we find that for a given search space the MOF cost function~\eqref{eqn:plic:mof:costfun} yields the smallest error.
We find that the MOF method~\citep{Dyadechko2005} indeed results in fifth order accuracy in the first moment, as stated in~\cref{eqn:plic:mof:fmaccuracy}.
Using PPIC compared to PLIC results in an increase of one order of accuracy when the LVIRA cost function~\eqref{eqn:plic:lvira:costfun} is used.
The PMOF method results in an increase of two orders of accuracy in approximating the first moment, when compared to the MOF method.
Note that the PMOF error stops decreasing around $h = 2^{-8}$, corresponding to a first moment error of $10^{-8} \times (2^{-8})^3 \approx 6 \times 10^{-16}$ which is very similar to our machine accuracy $\epsilon_M = 2^{-52} \approx 2 \times 10^{-16}$.

\subsection{Vortex reverse}\label{sec:results:reverse}
\begin{figure}
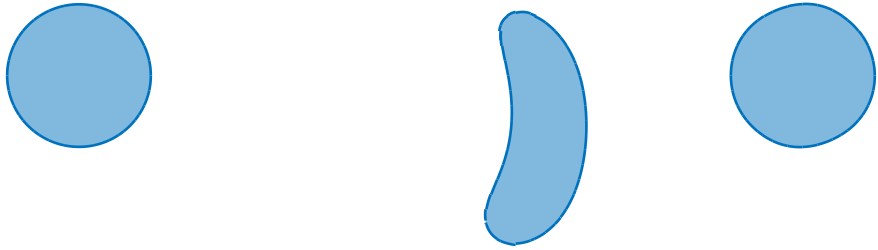

  \def\tikzWidth{\textwidth*0.3}
  \def\tikzHeight{\textwidth*0.3}
  \inputtikzorpdf{advection_example_plvira_tdx1_N32}
  \def\tikzWidth{\textwidth*0.3}
  \def\tikzHeight{\textwidth*0.3}
  \inputtikzorpdf{advection_example_plvira_tdx2_N32}
  \def\tikzWidth{\textwidth*0.3}
  \def\tikzHeight{\textwidth*0.3}
  \inputtikzorpdf{advection_example_plvira_tdx3_N32}
  \caption{The initial, intermediate and final approximate liquid domain at $t = 0, T/2$ and $t = T$ respectively, from left to right.
  Here the PLVIRA reconstruction method was used with $h / R \approx 1 / 5$.}
  \label{fig:int_validation_advection:example}
\end{figure}
The time-dependent accuracy of the proposed interface advection method is evaluated using the classical vortex reverse problem~\citep{Rider1997} where the interface undergoes a reversible deformation defined by the stream function
\begin{equation}
  \Psi = \frac{\cos(\pi t/T)}{\pi} \sin(\pi x)^2 \cos(\pi y)^2,
\end{equation}
for some period $T$.
Here the initial interface is a circle of radius $R = 0.15$ centred at $\+x_0 = \begin{bmatrix}0.5 & 0.75\end{bmatrix}^T$.
The accuracy of the Lagrangian remapping method is then measured using the area of the symmetric difference, as well as the $L^\infty$-norms of the approximation errors at $t = T$ of the first moment, the volume fraction (or normalised zeroth moment) as well as the curvature.
We let $T = 1$ such that it is relatively easy to resolve the interface shape, as shown in~\cref{fig:int_validation_advection:example}.

\begin{figure}
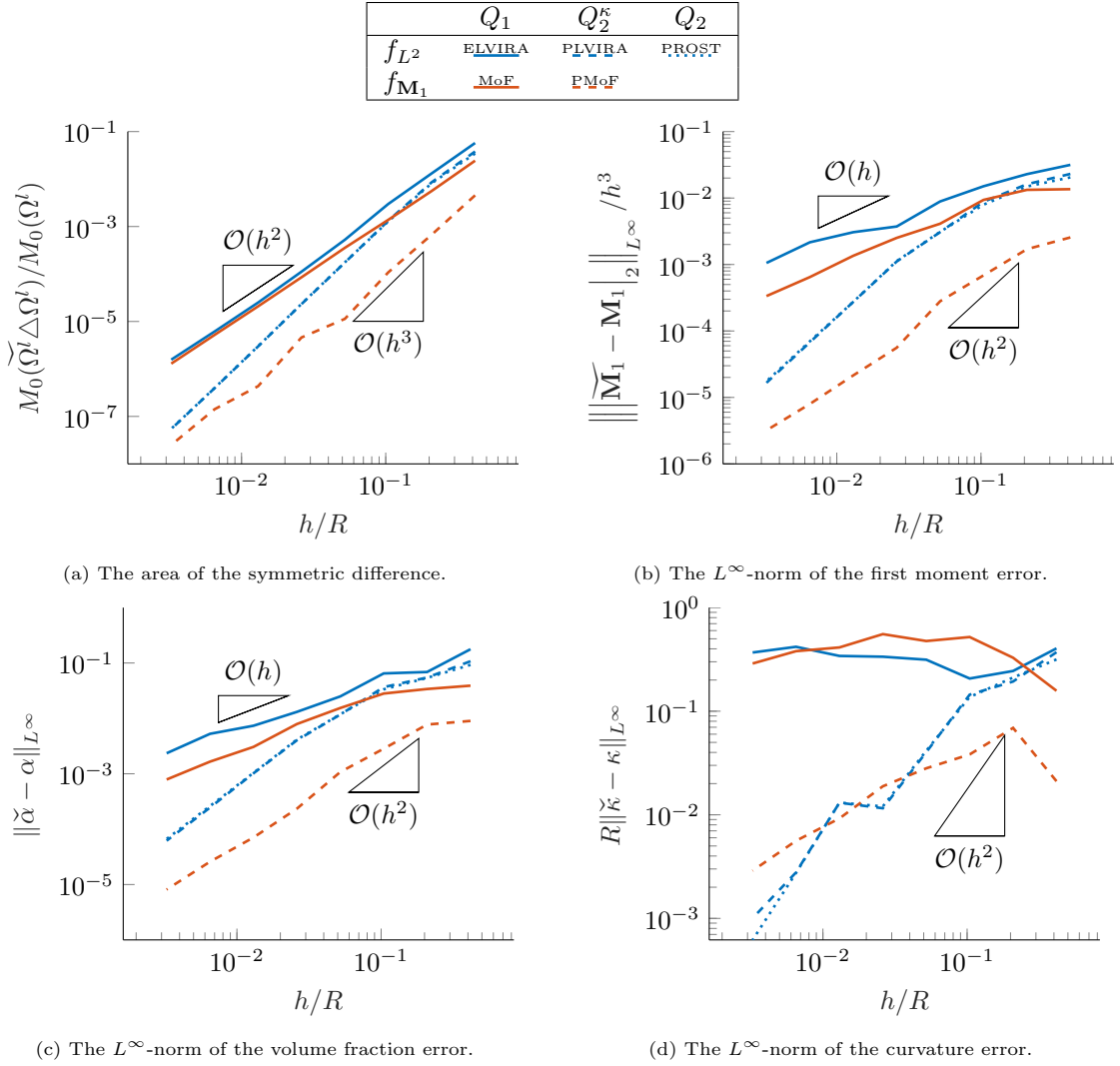

  \begin{subfigure}[t]{\textwidth}
    \centering
    \inputtikzorpdf{int_validation_optim_legend}
  \end{subfigure}
  \subcaptionbox{The area of the symmetric difference.
  \label{fig:int_validation_advection_vortex_T1_LEAS_path_sd}}
  [\twofigwidth]{
    \def\tikzWidth{\textwidth*0.35}
    \def\tikzHeight{\textwidth*0.3}
    \inputtikzorpdf{int_validation_advection_vortex_T1_LEAS_path_sd}
  }\hfill
  \subcaptionbox{The $L^\infty$-norm of the first moment error.
  \label{fig:int_validation_advection_vortex_T1_LEAS_path_firstMoment}}
  [\twofigwidth]{
    \def\tikzWidth{\textwidth*0.35}
    \def\tikzHeight{\textwidth*0.3}
    \inputtikzorpdf{int_validation_advection_vortex_T1_LEAS_path_firstMoment}
  }

  \subcaptionbox{The $L^\infty$-norm of the volume fraction error.
  \label{fig:int_validation_advection_vortex_T1_LEAS_path_LInf}}
  [\twofigwidth]{
    \def\tikzWidth{\textwidth*0.35}
    \def\tikzHeight{\textwidth*0.3}
    \inputtikzorpdf{int_validation_advection_vortex_T1_LEAS_path_LInf}
  }\hfill
  \subcaptionbox{The $L^\infty$-norm of the curvature error.
  \label{fig:int_validation_advection_vortex_T1_LEAS_path_kappa}}
  [\twofigwidth]{
    \def\tikzWidth{\textwidth*0.35}
    \def\tikzHeight{\textwidth*0.3}
    \inputtikzorpdf{int_validation_advection_vortex_T1_LEAS_path_kappa}
  }
  \caption{Accuracy of the interface shape for the vortex reverse problem at $t = T$.
  The accuracy is measured using the symmetric difference as well as the $L^\infty$-norm of the errors in the first moment, volume fraction (normalised zeroth moment) and curvature.
  Note that the results for PLVIRA and PROST essentially overlap.}
  \label{fig:int_validation_advection}
\end{figure}

The area of the symmetric difference, shown in~\cref{fig:int_validation_advection_vortex_T1_LEAS_path_sd}, shows the same orders of accuracy as found in the reconstruction test, which was to be expected since~\cref{thm:intro:fraction_error_bound} shows that the reconstruction error is the lowest order term in the total error estimate.
For the $L^\infty$-norm of the first moment error, we find that the accuracy for time-dependent problems is limited to fifth order, which can be expected from~\cref{eqn:plic:mof:steptwoaccuracy}.
As with the reconstruction test, we find that for a given search space, the MOF cost function consistently yields the most accurate first moment.

In~\cref{fig:int_validation_advection_vortex_T1_LEAS_path_LInf} we show the resulting $L^\infty$-norm of the volume fraction error at $t = T$.
We find that the PPIC methods result in an increase of one order of accuracy as compared to the PLIC methods.
Provided with the numerically observed reconstruction accuracies ($r = 3$ for PLIC and $r = 4$ for PPIC) we can understand these results using~\cref{thm:intro:fraction_error_bound}.

Furthermore, in~\cref{fig:int_validation_advection_vortex_T1_LEAS_path_kappa}, we show the accuracy of the curvature at $t = T$ in terms of the $L^\infty$-norm.
We find that the PLIC methods do not lead to a convergent curvature, as can be explained from the discussion in~\cref{sec:curvature}.
The PMOF method yields a first-order accurate curvature, which we expected from the aforementioned discussion.
Interestingly, the PLVIRA and PROST methods yield a second-order accurate curvature, which we cannot explain using the simple, and apparently pessimistic, arguments used in~\cref{sec:curvature}.

To facilitate any future comparison with our proposed methods, we have included the numerical values used in~\cref{fig:int_validation_advection} as supplementary material.

\begin{figure}
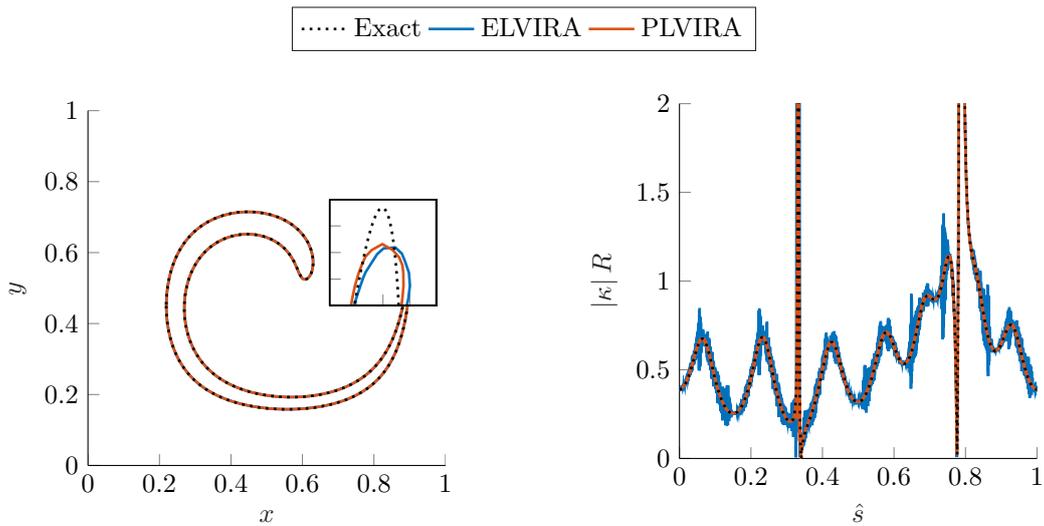

  \centering
  \color{black}
  \inputtikzorpdf{advection_long_legend}

  \subcaptionbox{Interface profile with an inset showing the region of largest curvature $\abs{\kappa} h \approx 5$.\label{fig:results:advection_long:profile}}
  [\twofigwidth]{
    \def\tikzWidth{\textwidth*0.32}
    \def\tikzHeight{\textwidth*0.32}
    \inputtikzorpdf{advection_long_profile_N1024_custom}
  }\hfill
  \subcaptionbox{Nondimensionalised curvature as function of the normalised arc length $\hat{s}$.\label{fig:results:advection_long:curvature}}
  [\twofigwidth]{
    \def\tikzWidth{\textwidth*0.32}
    \def\tikzHeight{\textwidth*0.32}
    \inputtikzorpdf{advection_long_curvature_N1024}
  }
  \caption{Solution to the vortex reverse problem with period $T = 4$ at $t = T / 2$.
  Here we used $h / R \approx 7 \times 2^{-10}$.}
\end{figure}
In order to test how the parabolic reconstruction methods respond to large deformations, we now repeat the same vortex reverse problem, but we now let $T = 4$.
The interface profile at $t = T/2$ is show in~\cref{fig:results:advection_long:profile}, where we only show the ELVIRA and PLVIRA methods to avoid cluttering the figure.
The inset shows a close up of the interface at the point of largest deformation. 
At this resolution we find $|\kappa|h \approx 5$, which means that the interface is unresolved at this point.
We find that, qualitatively, there is little difference between the interface profiles resulting from the ELVIRA and PLVIRA reconstruction methods.
The interface curvature as function of the normalised arc length $\hat{s}$ is shown in~\cref{fig:results:advection_long:curvature} at $t = T / 2$, which shows that the interface curvature resulting from the PLVIRA method is indistinguishable from the exact interface curvature, whereas the interface curvature resulting from the ELVIRA method is highly oscillatory and inaccurate.

\subsection{Galilean invariance and spurious currents}\label{sec:results:translation}
We now consider a droplet translation problem where the interface advection method is coupled to a two-phase Navier--Stokes solver.
The uniform velocity field is initialised as $\+u_0 = \+e_1 U$, and the initial interface corresponds to a droplet of radius $R = \SI{0.15}{m}$.
The domain is periodic with length $\SI{1}{m}$ and therefore at $t = k / U$ the droplet should have returned to its initial position, for any integer $k$.
As our proposed methods are thus far two-dimensional, and since our underlying Navier--Stokes solver is implemented in Cartesian coordinates, we are limited to considering a 2D droplet only.

The fluids are assumed inviscid and surface tension\footnote{Surface tension has been discretised using a well-balanced ghost fluid method~\citep{Popinet2018,Liu2000}. 
We make use of the GHF curvature approximation~\citep{Popinet2009}, except when the PROST method is used.} acts at the interface between the fluids, with surface energy coefficient $\sigma$.
The problem is thus fully characterised by the Weber number as well as by the density ratio, which are defined as
\begin{equation}
  \weber \defeq \frac{\rho^l U^2 R}{\sigma}, \quad \ratio{\rho} \defeq \frac{\rho^g}{\rho^l}.
\end{equation}
We fix the density ratio $\ratio{\rho} = 10^{-3}$ and consider three values of the Weber number: $\weber = 0$ corresponds to a stationary droplet (i.e., we consider the equilibrium rod as proposed by~\citet{Brackbill1992} for which $U = \SI{0}{m/s}$), $\weber = 1$ corresponds to a translating droplet with surface tension and $\weber = \infty$ corresponds to a translating droplet without surface tension (and thus uncoupled from the two-phase Navier--Stokes solver).
Note that in principle the initial velocity $\+u_0$ should leave the solution invariant up to translation since the two-phase Navier--Stokes equations are Galilean invariant.
It follows that the numerical solution should also become independent of the Weber number under mesh refinement.
For $\weber > 0$ we simulate until $k = 10$, and therefore stop the simulation at $t = 10/U$.
This means that the 2D droplet has passed through the periodic domain boundary $10$ times.

Coupling to the Navier--Stokes equations means that momentum is now transported as well.
For the transport of momentum we have implemented a momentum conserving convection scheme, very similar to the one presented in~\citep{Arrufat2021}, based on the advection method proposed by~\citet{Owkes2014}.
For this reason we will now use this advection method instead of the one discussed in~\cref{sec:remap:method}.

Based on the previous results, which favoured (in terms of curvature accuracy) the PPIC methods based on the LVIRA cost function, we now consider only the comparison between the ELVIRA, PLVIRA and PROST methods.
\begin{figure}
  \begin{subfigure}[t]{\textwidth}
    \centering
    % \tikzexternaldisable
    \inputtikzorpdf{st_validation_translating_droplet_legend}
    % \tikzexternalenable
  \end{subfigure}
  \subcaptionbox{Example energy evolution for $h / R \approx 1 / 38$.
  \label{fig:st_validation_translating_droplet_energy_example_N256}}
  [\twofigwidth]{
    \def\tikzWidth{\textwidth*0.33}
    \def\tikzHeight{\textwidth*0.3}
    \inputtikzorpdf{st_validation_translating_droplet_energy_example_N256_custom}
  }\hfill
  \subcaptionbox{Convergence at $t = L / U$.
  \label{fig:st_validation_translating_droplet_energy_convergence}}
  [\twofigwidth]{
    \def\tikzWidth{\textwidth*0.3}
    \def\tikzHeight{\textwidth*0.3}
    \inputtikzorpdf{st_validation_translating_droplet_energy_convergence}
  }
  \caption{The kinetic energy in the moving frame of reference.
  Note that for $\weber = \infty$ the kinetic energy in the moving frame of reference vanishes since there is no coupling via surface tension.}
  \label{fig:st_validation_translating_droplet_energy}

  \subcaptionbox{The area of the symmetric difference.
  \label{fig:st_validation_translating_droplet_symmdiff_convergence}}
  [\twofigwidth]{
    \def\tikzWidth{\textwidth*0.3}
    \def\tikzHeight{\textwidth*0.3}
    \inputtikzorpdf{st_validation_translating_droplet_symmdiff_convergence}
  }\hfill
  \subcaptionbox{The $L^\infty$-norm of the curvature error.
  \label{fig:st_validation_translating_droplet_kappa_convergence}}
  [\twofigwidth]{
    \def\tikzWidth{\textwidth*0.33}
    \def\tikzHeight{\textwidth*0.3}
    \inputtikzorpdf{st_validation_translating_droplet_kappa_convergence}
  }
  \caption{Convergence of the interface shape using several error measures for the droplet translation problem at $t = L / U$.}
  \label{fig:st_validation_translating_droplet} 
\end{figure}

In~\cref{fig:st_validation_translating_droplet_energy_example_N256} we show an example of the evolution of the kinetic energy in the frame of reference of the droplet for $\weber \in \{0, 1\}$ (the kinetic energy in the frame of reference of the droplet vanishes for $\weber = \infty$)
\begin{equation}
  E_k(t) = \half\integral{\Omega}{\rho(t, \+x) \abs{\+u(t, \+x) - \+u_0}_2^2}{V}.
\end{equation}
Time has been nondimensionalised using the oscillation period $T_\sigma(4)$ of a 2D droplet~\citep{Lamb1932}, where\footnote{In principle all modes will be triggered, however higher modes require more energy since the surface area increases with increasing $k$, hence we expect the lowest value of $k$ to be dominant.
Then, by symmetry of the domain, we find that $k=4$ is the lowest mode which has the same symmetries as the square domain.}
\begin{equation}
  T_\sigma(k) = 2 \pi\squarepar{\frac{R^3 (\rho^g + \rho^l)}{\sigma k (k^2 - 1)}}^{1/2}.
\end{equation}
The kinetic energy has two maxima per oscillation period and therefore we find oscillations in the kinetic energy of period $T_\sigma(4)/2$.

We note that the PLVIRA method yields a kinetic energy evolution which is essentially independent of the Weber number, as expected from the Galilean invariance of the Navier--Stokes equations.
Moreover we find that the resulting oscillation in kinetic energy is of the correct physical frequency, and is therefore absent of spurious currents, regardless of the Weber number.
For the ELVIRA method we find that for $\weber = 1$, despite using a well-balanced surface tension method, spurious kinetic energy is present due to the insufficiently accurate curvature.
Differences between the PROST and PLVIRA method finally start to show: for both finite Weber numbers, spurious kinetic energy is generated when the PROST method is used.
% \footnote{In particular, for the stationary droplet the capillary time-step restriction will result in $\cfl \rightarrow 0$ (hardly dissipative) whereas for the translating droplet, the time-step restriction will initially be dominated by the CFL constraint, resulting in $\cfl = \half$ (very dissipative).}
Furthermore we show the maximum kinetic energy in~\cref{fig:st_validation_translating_droplet_energy_convergence}, which shows that the ELVIRA and PROST methods do not converge in terms of the maximum kinetic energy (which should vanish under mesh refinement) for $\weber = 1$.
The PLVIRA method does converge and shows convergence similar that of $\weber = 0$.

Finally, in~\cref{fig:st_validation_translating_droplet} we show convergence of the interface shape using the area of the symmetric difference as well as the $L^\infty$-norm of the curvature error at $t = L / U$.
In terms of the symmetric difference we find that PLVIRA is the most accurate method for any of the Weber numbers.
As for the curvature error, we find that only the PLVIRA method exhibits convergence of the curvature regardless of the Weber number, which is in agreement with the observation that only the PLVIRA method results in a convergent kinetic energy regardless of the Weber number.

\subsection{Rising bubble}
As a final test case we consider the rising of a bubble due to buoyancy.
Since our numerical method is (thus far) two dimensional, and since the implementation is restricted to Cartesian coordinate systems, we are limited to considering a 2D rising bubble.
We will use the 2D benchmark proposed in~\citet{Hysing2009} (therein referred to as case 2), where an initially circular bubble of radius $R = \SI{0.25}{m}$ is considered at $\+x_0 = [0.5, 0.5]^T$ inside a domain given by $\Omega = [0, 1] \times [0, 2]$.
The fluid properties are given by $\rho^l = \SI{e3}{kg/m^3}, \rho^g = \SI{1}{kg/m^3}, \mu^l = \SI{10}{Pa.s}, \mu^g = \SI{e-1}{Pa.s}, \sigma = \SI{1.96}{J/m^2}$, and the gravitational acceleration is given by $g_z = \SI{-0.98}{m/s^2}$.

\begin{figure}
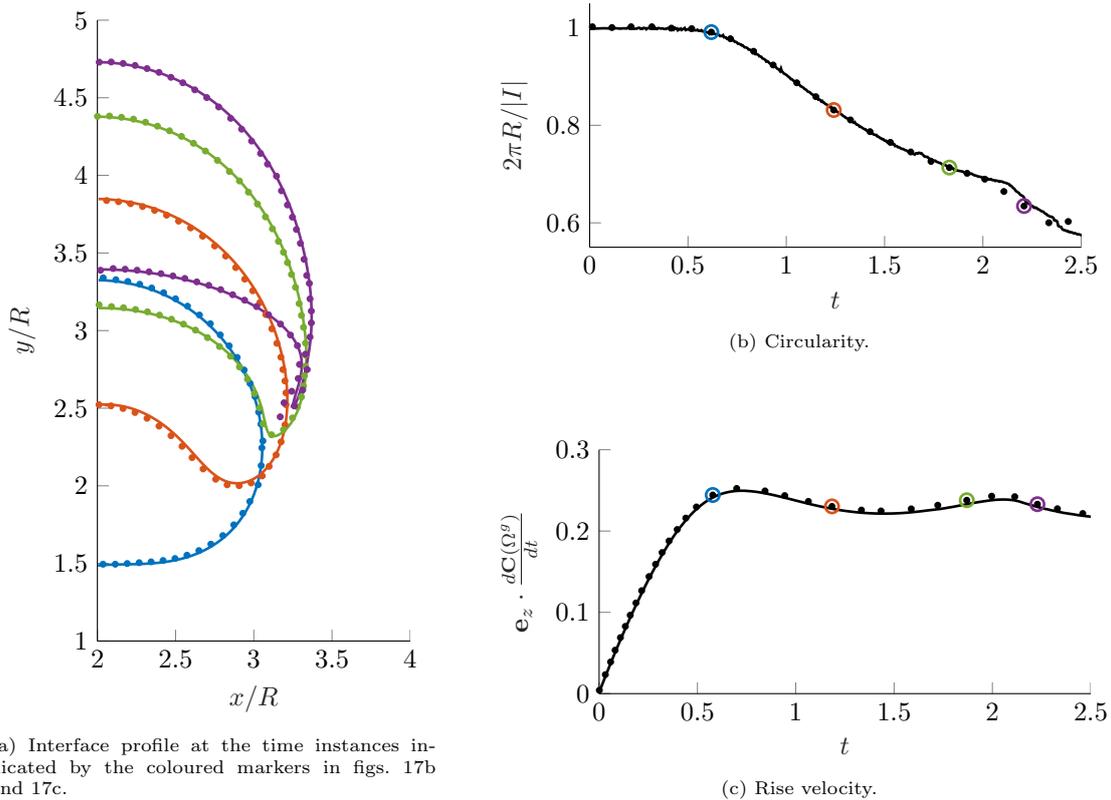

  \centering
  \color{black}
  
  \valign{#\cr
    \hsize=0.4\textwidth
    \begin{subfigure}{0.4\textwidth}
      \centering
      \def\tikzWidth{\textwidth*0.7}
      \def\tikzHeight{\textwidth*1.4}
      \inputtikzorpdf{rising_bubble_profile_plvira_kwamr6}
      \caption{Interface profile at the time instances indicated by the coloured markers in~\cref{fig:results:bubble:circularity,fig:results:bubble:velocity}.}\label{fig:results:bubble:profile}
    \end{subfigure}\cr\noalign{\hfill}
    \hsize=0.55\textwidth
    \begin{subfigure}{0.55\textwidth}
      \centering
      \def\tikzWidth{\textwidth*0.8}
      \def\tikzHeight{\textwidth*0.4}
      \inputtikzorpdf{rising_bubble_circularity_plvira_kwamr6}
      \caption{Circularity.}\label{fig:results:bubble:circularity}
    \end{subfigure}\vfill
    \begin{subfigure}{0.55\textwidth}
      \centering
      \def\tikzWidth{\textwidth*0.8}
      \def\tikzHeight{\textwidth*0.4}
      \inputtikzorpdf{rising_bubble_velocity_plvira_kwamr6}
      \caption{Rise velocity.}\label{fig:results:bubble:velocity}
    \end{subfigure}\cr
  }
  \caption{Comparison of the proposed PLVIRA method using curvature and vorticity based adaptive mesh refinement (solid lines) to a reference solution (markers) obtained from~\citet{Hysing2009}, resulting in a smallest mesh width of $h / R = 2^{-8}$ (with $R = \SI{0.25}{m}$).}
\end{figure}
Based on the previously presented results, which clearly favoured the PLVIRA reconstruction method, we will consider only this method combined with curvature and vorticity based adaptive mesh refinement, resulting in a smallest mesh width of $h / R = 2^{-8}$.
Results of the interface profile, circularity (initial interface length divided by interface length), as well as the rise velocity (time derivative of the gas centroid), compared to the reference solution from~\citet{Hysing2009} (we use the solution indicated by `TP2D $h = 1/640$', which uses a finite element discretisation of the Navier--Stokes equations, and is coupled to a level set method), are shown in~\cref{fig:results:bubble:profile,fig:results:bubble:circularity,fig:results:bubble:velocity} respectively.
Good agreement is obtained between our proposed solution and the reference solution.

\subsection{Computational cost}
It should be noted that the computational cost for the optimisation based methods will depend on the choice of numerical optimisation method, as well as the method by which the derivatives are computed. 
We use an exact linearisation of the cost functions when applied to a parabolic search space, as explained in~\cref{sec:app:derivative_mof,sec:app:derivative_l2} for the MOF and LVIRA cost function respectively.
We have also included the LVIRA method for comparison.
\begin{figure}
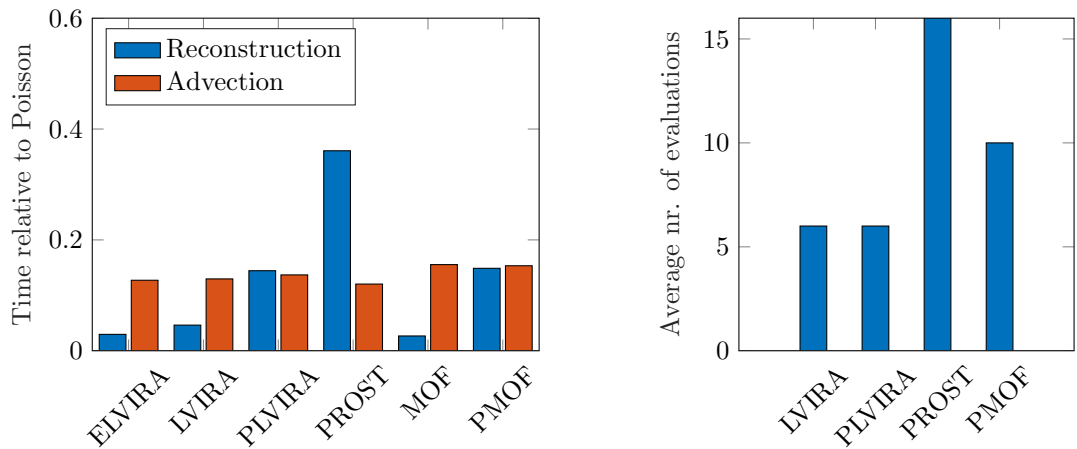

  \subcaptionbox{The wall-clock time relative to the time needed to solve the pressure Poisson problem.
  \label{fig:int_validation_computational_cost}}
  [0.525\textwidth]{
    \def\tikzWidth{\textwidth*0.4}
    \def\tikzHeight{\textwidth*0.3}
    \inputtikzorpdf{int_validation_computational_cost_LEAS_N128}
  }\hfill
  \subcaptionbox{Average number of cost function evaluations per interface reconstruction.
  \label{fig:int_validation_function_evaluations}}
  [0.425\textwidth]{
    \def\tikzWidth{\textwidth*0.3}
    \def\tikzHeight{\textwidth*0.3}
    \inputtikzorpdf{int_validation_function_evaluations_LEAS_N128}
  }
  \caption{Comparison of the reconstruction cost for several reconstruction methods.}
  \label{fig:int_validation_cost}
\end{figure}

In~\cref{fig:int_validation_computational_cost} we show the computational time relative to the time needed to solve the pressure Poisson problem for each of the reconstruction methods for a problem coupled to the two-phase Navier--Stokes solver.
For comparison we also show the computational time of the advection method, which is essentially the time needed to approximate the preimage, and compute the intersection volume given by~\cref{eqn:remap:approx:zerothmoment}.
Moreover, in~\cref{fig:int_validation_function_evaluations} we show the average number of cost function evaluations per reconstruction.

We find that using the ELVIRA method instead of the LVIRA method yields a slight reconstruction time reduction, confirming that the `E' in the acronym ELVIRA is justified.
Going from LVIRA to PLVIRA yields about a factor three in reconstruction time, which is acceptable since the reconstruction cost for the PLVIRA method is comparable to the cost of the advection method and still small when compared to the the time needed to solve the pressure Poisson problem.

The PROST method is more than twice as expensive as the PLVIRA method, which is most likely due to the fact that there are two unknowns in the optimisation problem, rather than only one for PLVIRA.

The PMOF method is about as expensive as the PLVIRA method, while requiring more function evaluations.
The fact that more iterations are required can likely be explained by the observation that the second derivative of the cost function is much larger (see for example~\cref{fig:int_reconstruction_costFun}), which makes it harder to find the local minimum.
On the other hand the evaluation of the cost function $f_{\+M_1}$ is cheaper than the evaluation of the cost function $f_{L^2}$ (used by PLVIRA), which results in the reconstruction time being comparable to that of PLVIRA.

\section{Conclusion}\label{sec:conclusion}
We showed that traditional geometric reconstruction methods, that are based on a piecewise linear approximation of the interface, do not yield convergence under mesh refinement of the GHF curvature for time-dependent advection problems.
Moreover we showed that the leading order error term of a Lagrangian remapping method is due to the PLIC reconstruction method, and have argued that this leading order term, which is first order in $h$, directly results in the observed lack of convergence of the curvature.

We have therefore presented two new geometric interface reconstruction methods, PLVIRA and PMOF, both examples in a class of PPIC methods which are based on a piecewise parabolic approximation of the interface.
The two methods are generalisations of the LVIRA and MOF reconstruction methods, the generalisation being in the search space: going from a space of linear interfaces to one of parabolic interfaces.
The corresponding cost functions have remained unchanged.

We have shown through numerical experiments that the proposed PPIC methods perform favourably when compared to the PLIC methods: the reconstruction accuracy is increased, resulting in a convergent curvature for time-dependent advection problems.

Furthermore, we have demonstrated that the PLVIRA method results in Weber number independent convergence for the droplet translation problem, contrary to the ELVIRA method.
Only for the droplet translation problem, where the interface advection method is coupled to a two-phase Navier--Stokes solver, do we observe different results for the PLVIRA and PROST methods: only the former results in Weber number independent convergence.
Exactly why the PROST method differs from the PLVIRA method on this point requires further investigation.

A test case involving a rising bubble was considered to demonstrate the applicability of the proposed interface reconstruction methods.
The resulting interface profile, circularity as well as rise velocity were found to be in good agreement with a reference solution.

The work presented here was limited to two spatial dimensions, obviously 3D simulations are of interest and therefore the geometric methods proposed here, in particular the computation of moments of the intersection of a polygon with a parabola, must be extended to 3D.
In~\citet{Renardy2002} a method for approximately computing the zeroth moment in 3D is proposed, they claim that this method yields second-order accurate volume fractions and should therefore be sufficiently accurate according to~\cref{thm:intro:fraction_error_bound}.

It would be interesting to have a full theoretical understanding of the reconstruction error; that is to show that $E^\mtext{Rec}_0 = \mathcal{O}(h^4)$ for the proposed PPIC methods.
Such an understanding could also provide a stability result (accuracy w.r.t. perturbations in the reference moments) that can be used to extend the single time-step consistency result of~\cref{thm:intro:fraction_error_bound} to a multiple time-step convergence result.

\begin{minipage}{.9\textwidth}
  \centering
  \vspace{1em}
  \textbf{Acknowledgements}
  \\
  This work is part of the research programme SLING, which is (partly) financed by the Netherlands Organisation for Scientific Research (NWO).
\end{minipage}

\setcounter{figure}{0}  
\appendix
\section{Proofs of the remapping error estimates}\label{sec:remap_accuracy}
% \lemmadecomposition*
\setcounter{lemma}{0}
\setcounter{table}{0}
\begin{lemma}[Error decomposition]
  
\end{lemma}
\begin{proof}
  The proof relies on the following three properties
  \begin{align}
    \abs{M_0(A) - M_0(B)} &\le M_0(A \symmdiff B)\label{eqn:app:property_1}\\
    A \symmdiff B &\subseteq (A \symmdiff C) \cup (C \symmdiff B)\label{eqn:app:property_2}\\
    (A \cap B) \symmdiff (C \cap D) &\subseteq (A \symmdiff C) \cup (B \symmdiff D)\label{eqn:app:property_3}
  \end{align}
  where $A, B, C, D \subseteq \mathbb{R}^2$.
  We will first prove each of these properties

  \begin{enumerate}
    \item[\ref{eqn:app:property_1}] Note that
    \begin{equation}
      \abs{M_0(A) - M_0(B)} = \abs{\integral{\mathbb{R}^2}{(\chi_A - \chi_B)}{V}} \le \integral{\mathbb{R}^2}{\abs{\chi_A - \chi_B}}{V} = M_0(A \symmdiff B),
    \end{equation}
    where $\chi_A, \chi_B$ are the indicator functions of the sets $A, B$ respectively.
    This proves~\cref{eqn:app:property_1}.
    \item[\ref{eqn:app:property_2}] Since the symmetric difference is both commutative and associative, it follows that
    \begin{equation}
      (A \symmdiff C) \symmdiff (C \symmdiff B) = (A \symmdiff B) \symmdiff (C \symmdiff C) = (A \symmdiff B) \symmdiff \emptyset = A \symmdiff B,
    \end{equation}
    from which~\cref{eqn:app:property_2} follows (the triangle inequality)
    \begin{equation}
      A \symmdiff B \subseteq (A \symmdiff C) \cup (C \symmdiff B).
    \end{equation}
    \item[\ref{eqn:app:property_3}] Note that the symmetric difference can also be written as
    \begin{equation}
      A \symmdiff B = (A \cap B^\complement) \cup (A^\complement \cap B),
    \end{equation}
    where $A^\complement$ denotes the complement of $A$ in $\mathbb{R}^2$.
    The proof is as found in~\citet[prop. 3.5]{Zhang2013}
    \begin{align}
      (A \cap B) \symmdiff (C \cap D) &= (A \cap B \cap (C \cap D)^\complement) \cup ((A \cap B)^\complement \cap C \cap D)\\
      &= (A \cap B \cap (C^\complement \cup D^\complement)) \cup ((A^\complement \cup B^\complement) \cap C \cap D)\\
      &= (A \cap B \cap C^\complement) \cup (A \cap B \cap D^\complement) \cup (A^\complement \cap C \cap D) \cup (B^\complement \cap C \cap D)\\
      &\subseteq (A \cap C^\complement) \cup (B \cap D^\complement) \cup (A^\complement \cap C) \cup (B^\complement \cap D)\\
      &= \squarepar{(A \cap C^\complement) \cup (A^\complement \cap C)} \cup \squarepar{(B \cap D^\complement)  \cup (B^\complement \cap D)}\\
      &= (A \symmdiff C) \cup (B \symmdiff D).
    \end{align}
  \end{enumerate}

  From~\cref{eqn:app:property_1} it follows that
  \begin{equation}
    \abs{\approximate{M_{0,c}^l} - M_{0,c}^l} = \abs{M_0\roundpar{\approximate{P_c} \cap \approximate{\Omega_c^{l}}} - M_0\roundpar{P_c \cap \Omega_c^{l}}} \le M_0\roundpar{\roundpar{\approximate{P_c} \cap \approximate{\Omega_c^{l}}} \symmdiff \roundpar{P_c \cap \Omega_c^{l}}}.
  \end{equation}
  Then, using~\cref{eqn:app:property_3} shows that
  \begin{equation}
    \roundpar{\approximate{P_c} \cap \approximate{\Omega_c^{l}}} \symmdiff \roundpar{P_c \cap \Omega_c^{l}}  \subseteq \roundpar{\approximate{\Omega_c^{l}} \symmdiff \Omega_c^{l}} \cup \roundpar{\approximate{\preimage_c} \symmdiff \preimage_c},
  \end{equation}
  from which it follows that
  \begin{equation}
    \abs{\approximate{M_{0,c}^l} - M_{0,c}^l} \le M_0\roundpar{\approximate{\Omega_c^{l}} \symmdiff \Omega_c^{l}} + M_0\roundpar{\approximate{\preimage_c} \symmdiff \preimage_c}.
  \end{equation}
  Finally, we use~\cref{eqn:app:property_2} (with $A = \approximate{\preimage_c}, B = \preimage_c$ and $C = \preimage^\mtext{Rep}_c$), resulting in
  \begin{equation}
    \abs{\approximate{M_{0,c}^l} - M_{0,c}^l} \le M_0\roundpar{\approximate{\Omega_c^{l}} \symmdiff \Omega_c^{l}} + M_0\roundpar{\approximate{\preimage_c} \symmdiff \preimage^\mtext{Rep}_c} + M_0\roundpar{\preimage^\mtext{Rep}_c \symmdiff \approximate{\preimage_c}}.
  \end{equation}
\end{proof}

Lemma~\ref{lem:error_remapping} is similar to proposition 3.3 of~\citet{Zhang2013}, with the most notable difference that we consider a remapping method rather than a donating region method and moreover we derived an improved error estimate of $E_0^\mtext{Rep}$ (which~\citet{Zhang2013} refers to as $E^\mtext{Rep}_\mtext{image}$) by making use of the approximate linearity preservation of the flow map.
Furthermore we use approximate pathline integration rather than approximate streamline integration.

% \lemmaremapping*
\begin{lemma}[Remapping error estimates]
  
\end{lemma}
\begin{proof}
  \newcommand\approxRep[1]{{#1}^\mtext{Rep}}
  Let $\+p_0, \+p_1$ be the endpoints of an edge of $\partial c$, we parametrise the edge in the following way
  \begin{equation}
    \+p_s(\tau) = \Psi^\tau\roundpar{s \+p_1 + (1-s) \+p_0}, \quad (\tau, s) \in [-\delta, 0] \times [0, 1].
  \end{equation}
  Similarly, we parametrise the polygonal representation of the edge as
  \begin{equation}
    \approxRep{\+p}_s(\tau) = s \+p_1(\tau) + (1-s) \+p_0(\tau), \quad (\tau, s) \in [-\delta, 0] \times [0, 1],
  \end{equation}
  see also~\cref{fig:app:remap:representation_error}.
  \begin{figure}
    \centering
    \import{inkscape/}{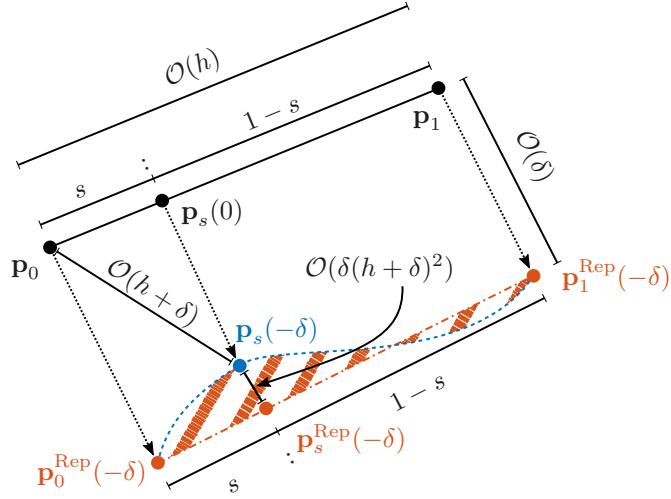}
    \caption{
      Illustration of the representation error for a single edge $e \subset \partial c$ of the control volume $c$, where the edge is defined by its endpoints $\+p_0, \+p_1$.
      The dotted arrows indicate the application of $\Psi^{-\delta}$ (exact integration backwards in time).
      The dashed curve corresponds to the exact preimage of the edge (given by $\+p_s(-\delta)$ for $s \in [0, 1]$), and the dash-dotted line corresponds to the polygonal representation (given by $\approxRep{\+p}_s(-\delta)$ for $s \in [0, 1]$).
      The area of the hatched region contributes to the representation error $E_0^\mtext{Rep}$ (for the total representation error we sum over all of the edges of the boundary of the control volume).
      See also~\cref{fig:remap:remapping_errors_rep}.
    }
    \label{fig:app:remap:representation_error}
  \end{figure}
  Note that the polygonal representation satisfies the following ODE
  \begin{equation}
    \frac{d}{d\tau} \approxRep{\+p}_s(\tau) = \approxRep{\+u}_s(\tau) \defeq \+u\roundpar{t^{n+1}+\tau, \+p_0} + \gradient \+u\roundpar{t^{n+1}+\tau, \+p_0} (\approxRep{\+p}_s(\tau) - \+p_0),
  \end{equation}
  from which it follows that the velocity $\approxRep{\+u}_s(\tau)$ agrees with the exact velocity field in the first \emph{two terms} of its Taylor series expansion around $\+x = \+p_0$ (this is what we refer to as `approximate linearity preservation of the flow map').
  It follows that, at $\+x = \+p_s(\tau)$, the difference between the two velocities results in a second-order term
  \begin{equation}\label{eqn:app:taylor}
    \abs{\approxRep{\+u}_s(\tau) - \+u(t+\tau, \+p_s(\tau))}_2 = \mathcal{O}(\abs{\+p_s(\tau) - \+p_0}^2_2) = \mathcal{O}((h+\delta)^2).
  \end{equation}
  Integration of~\cref{eqn:app:taylor} from $\tau = 0$ to $\tau = -\delta$ then yields
  \begin{equation}\label{eqn:app:linearity_preservation}
    \abs{\approxRep{\+p}_s(-\delta)  - \+p_s(-\delta)}_2 = \mathcal{O}(\delta (h+\delta)^2).
  \end{equation}
  Hence the area enclosed by the preimages $\approxRep{\+p}_s(-\delta)$ and $\+p_s(-\delta)$, for $s \in [0, 1]$, is bounded by $h$ times the estimate provided in~\cref{eqn:app:linearity_preservation}
  \begin{equation}
    E_0^\mtext{Rep} = \mathcal{O}(h \delta (h + \delta)^2).
  \end{equation}

  \newcommand\approxInt[1]{{#1}^\mtext{Int}}
  Let $\approxInt{\+x}$ denote the exactly integrated position resulting from approximating $\+u(t, \+x)$ by a linearly interpolated (in space and time) velocity field $\approxInt{\+u}(t, \+x)$.
  Using linear interpolation implies that, for $t \in [t^{(n+1)}, t^{(n)}]$
  \begin{equation}
    \abs{\approxInt{\+u}(t, \+x) - \+u(t, \+x)}_2 = \mathcal{O}((h+\delta)^2),
  \end{equation}
  which upon integration results in
  \begin{equation}
    \abs{\approxInt{\+x}(t^{(n)}) - \+x(t^{(n)})}_2 = \mathcal{O}(\delta (h+\delta)^2).
  \end{equation}
  Moreover if the time integration is approximated using a $q$-th order accurate method, resulting in the fully approximated position $\approximate{\+x}$, we find
  \begin{equation}
    \abs{\approximate{\+x}(t^{(n)}) - \+x(t^{(n)})}_2 \le \abs{\approximate{\+x}(t^{(n)}) - \approxInt{\+x}(t^{(n)})}_2 + \abs{\approxInt{\+x}(t^{(n)}) - \+x(t^{(n)})}_2 = \mathcal{O}(\delta^{q+1}) + \mathcal{O}(\delta (h+\delta)^2).
  \end{equation}
  It follows that the integration error can be estimated by
  \begin{equation}
    E_0^\mtext{Int} = \mathcal{O}(h \delta (h + \delta)^2 + h\delta^{q+1}).
  \end{equation}
\end{proof}

\section{Initial interval for the computation of the shift}\label{sec:app:brent_bound}
The use of Brent's method for finding the root of $g(\phi)$ (as defined in~\cref{eqn:para:enforcement:g}) requires an initial interval $[\phi_-, \phi_+]$ for which $g(\phi_-) g(\phi_+) < 0$.

As first initial guess we use $\phi_1 = \phi(\+\eta; M_{0,c}^l)$, i.e. we assume that the curvature vanishes.
Suppose that $g(\phi_1) > 0$, which happens if and only if $\kappa < 0$, we then require an $\phi_2 < \phi_1$ (hence $\phi_- = \phi_2, \phi_+ = \phi_1$) with $g(\phi_2) \le 0$.
If we can find a shift $\phi_2$ for which the reconstructed liquid domain has empty intersection with the cell $c$
\begin{equation}\label{eqn:para:enforce:emptyset}
  c \cap \halfspace{\+\eta\cdot(\+x - \+x_c) - \phi_2 + \frac{\kappa}{2}(\+\tau\cdot(\+x - \+x_c))^2} = \emptyset,
\end{equation}
then we find that $g(\phi_2) = 0 - M_{0,c}^l \le 0$.
Note that~\cref{eqn:para:enforce:emptyset} can also be written as
\begin{equation}
  \set{\+x \in c}{\+\eta\cdot(\+x - \+x_c) - \phi_2 + \frac{\kappa}{2}(\+\tau\cdot(\+x - \+x_c))^2 \le 0} = \emptyset,
\end{equation}
from which we conclude that
\begin{equation}
  \phi_2 \le \min_{\+x \in c} \roundpar{\frac{\kappa}{2}(\+\tau\cdot(\+x - \+x_c))^2 + \+\eta\cdot(\+x - \+x_c)},
\end{equation}
is a sufficient condition for~\cref{eqn:para:enforce:emptyset} to hold.
We now let $\phi_\eta^* = \roundpar{h_x|\eta_x| + h_y|\eta_y|}/{2}, \phi_\tau^* = \roundpar{h_x|\tau_x| + h_y|\tau_y|}/{2}$ (assuming a rectangular control volume with dimensions $h_x \times h_y$) and note that since $\kappa < 0$
\begin{equation}
  \frac{\kappa}{2}(\+\tau\cdot(\+x - \+x_c))^2 + \+\eta\cdot(\+x - \+x_c) \ge \frac{\kappa}{2}(\phi_\tau^*)^2 - \phi_\eta^*, \quad \+x \in c.
\end{equation} 
It follows that
\begin{equation}
  \phi_2 = \frac{\kappa}{2}(\phi_\tau^*)^2 - \phi_\eta^*,
\end{equation}
always results in $g(\phi_2) \le 0$. 

Similarly if $g(\phi_1) < 0$, which happens if and only if $\kappa > 0$, we require an $\phi_2 > \phi_1$ (hence $\phi_- = \phi_1, \phi_+ = \phi_2$) with $g(\phi_2) \ge 0$.
The following value of $\phi_2$ is always suitable
\begin{equation}
  \phi_2 = \frac{\kappa}{2}(\phi_\tau^*)^2 + \phi_\eta^* \ge \max_{\+x \in c} \roundpar{\frac{\kappa}{2}(\+\tau\cdot(\+x - \+x_c))^2 + \+\eta\cdot(\+x - \+x_c)},
\end{equation}
since for this value we find that $g(\phi_2) = M_0(c) - M_{0,c}^l \ge 0$.

\def\interfacePart{I_c(q)}
\def\interfaceLength{|\interfacePart|}
\section{Linearisation of cost functions}\label{sec:app:linearisations}
Here we will discuss the linearisation of the cost functions $f^2_{L^2}$ and $f_{\+M_1}^2$, as given by~\cref{eqn:plic:lvira:costfun,eqn:plic:mof:costfun}.
Each of the cost functions is defined in terms of the zeroth or first moment of the intersection of a polygon and the surrounding fluid, and therefore the derivative of each of the cost functions requires the derivative of this zeroth or first moment w.r.t. the normal angle $\vartheta$.
For the PROST method we also require the derivative of the zeroth moment w.r.t. the curvature $\kappa$.
We will now explain the computation of these derivatives, and note that an implementation is provided in~\citet{intersection_alg}.

\subsection{Derivative of the MOF cost function}\label{sec:app:derivative_mof}
Given the derivative of the first moment w.r.t. the normal angle $\vartheta$, the derivative of the cost function can be obtained using the chain rule.
In~\citet{Dyadechko2005} the derivative of the first moment w.r.t. the normal angle $\vartheta$ for a linear interface in a convex control volume is given as
\begin{equation}\label{eqn:app:derivative:plic}
  \frac{d}{d\vartheta} \+M_1(c \cap \halfspace{q}) = \frac{1}{12} \interfaceLength^3 \+\tau, \quad q \in Q_1,
\end{equation}
where $\interfaceLength$ denotes the length of the interface contained in the control volume $c$
\begin{equation}
  \interfacePart \defeq \set{\+x \in c}{q(\+x) = 0}.
\end{equation}
Note that the level set $q$ depends, via the normal $\+\eta = \begin{bmatrix}\cos \vartheta & \sin \vartheta\end{bmatrix}^T$ and tangent $\+\tau = \frac{d}{d\vartheta} \+\eta$, on the angle $\vartheta$.
The authors of~\citep{Dyadechko2005} give~\cref{eqn:app:derivative:plic} without proof or reference, and therefore the generalisation to $q \in Q_2$ requires us to start from scratch.

The Reynolds transport theorem yields the following expression for the derivative of the first moment
\begin{equation}\label{eqn:app:derivative:firstmomentder}
  \frac{d}{d\vartheta} \+M_1(c \cap l(q)) = \frac{d}{d\vartheta} \integral{c \cap l(q)}{\+x}{V} = \integral{\interfacePart}{(\hat{\+\eta} \cdot \+v_\vartheta) \+x}{S},
\end{equation}
where $\+v_\vartheta$ denotes the derivative of the interface position w.r.t. the normal angle $\vartheta$ and $\hat{\+\eta}$ denotes the local interface normal which is given by
\begin{equation}
  \hat{\+\eta} = \frac{\gradient q}{\abs{\gradient q}_2}, \quad \gradient q = \+\eta + \kappa (\+\tau \cdot (\+x - \+x_c)) \+\tau.
\end{equation}
Similarly, we obtain the derivative of the zeroth moment
\begin{equation}\label{eqn:app:derivative:volumeder}
  \frac{d}{d\vartheta} M_0(c \cap l(q)) = \frac{d}{d\vartheta} \integral{c \cap l(q)}{}{V} = \integral{\interfacePart}{\hat{\+\eta} \cdot \+v_\vartheta}{S},
\end{equation}
which is imposed to vanish since interface reconstruction is volume preserving.

To be able to further manipulate the integrals in~\cref{eqn:app:derivative:firstmomentder,eqn:app:derivative:volumeder} we introduce the following explicit parametrisation of a parabolic interface
\begin{equation}\label{eqn:app:derivative:explicit_interface}
  \+p(\tau) = \+x_c + \roundpar{\phi - \frac{\kappa}{2} \tau^2} \+\eta + \tau \+\tau,
\end{equation}
where $\phi = \phi(\vartheta, \kappa; M^l_{0,c})$.
The parametrisation~\cref{eqn:app:derivative:explicit_interface} provides an explicit expression for $\+v_\vartheta$
\begin{equation}\label{eqn:app:derivative:pos_der_varphi}
  \+v_\vartheta = \frac{d}{d\vartheta} \+p = \roundpar{\frac{d\phi}{d\vartheta}-\tau} \+\eta + \roundpar{\phi - \frac{\kappa}{2} \tau^2} \+\tau,
\end{equation}
where we have made use of $\frac{d}{d\vartheta}\+\eta = \+\tau$ and $\frac{d}{d\vartheta}\+\tau = -\+\eta$.

We now evaluate~\cref{eqn:app:derivative:volumeder} using this parametrisation, by first transforming the line integral over $\interfacePart$ to a sum of one-dimensional integrations over parts of the interface, each parametrised by $\tau$
\begin{equation}\label{eqn:app:derivative:intermediate_result}
  \frac{d}{d\vartheta} M_0(c \cap l(q)) = \integral{\interfacePart}{\hat{\+\eta} \cdot \+v_\vartheta}{S} = \sum_{\hat{e} \in \hat{\mathcal{E}}^*} \oneDIntegral{\hat\tau_l}{\hat\tau_r}{\hat{\+\eta} \cdot \squarepar{\roundpar{\frac{d\phi}{d\vartheta}-\tau} \+\eta + \roundpar{\phi - \frac{\kappa}{2} \tau^2} \+\tau}\abs{\frac{d\+p}{d\tau}}_2}{\tau},
\end{equation}
where $\hat{\mathcal{E}}^*$ is as introduced in~\cref{sec:parabolic:intersect}.
Note that
\begin{equation}\label{eqn:app:derivative:explicit_interface_der}
  \frac{d}{d\tau} \+p = \+\tau - \kappa\tau \+\eta \quad \Rightarrow \quad \abs{\frac{d}{d\tau} \+p}_2 = \sqrt{1 + (\kappa\tau)^2} = \abs{\gradient q}_2 \quad \Rightarrow \quad \hat{\+\eta}\abs{\frac{d\+p}{d\tau}}_2 = \gradient q,
\end{equation}
and therefore
\begin{align}
  \frac{d}{d\vartheta} M_0(c \cap l(q)) &= \sum_{\hat{e} \in \hat{\mathcal{E}}^*} \oneDIntegral{\hat\tau_l}{\hat\tau_r}{\squarepar{\+\eta + \kappa \tau \+\tau} \cdot \squarepar{\roundpar{\frac{d\phi}{d\vartheta}-\tau} \+\eta + \roundpar{\phi - \frac{\kappa}{2} \tau^2} \+\tau}}{\tau}\\
  &= \sum_{\hat{e} \in \hat{\mathcal{E}}^*} \oneDIntegral{\hat\tau_l}{\hat\tau_r}{\squarepar{\frac{d\phi}{d\vartheta} - \tau + \kappa \tau \roundpar{\phi - \frac{\kappa}{2} \tau^2}}}{\tau}\label{eqn:app:derivative:zeroth_moment_derivative}
\end{align}
Conservation of volume ($\frac{d}{d\vartheta} M_0(c \cap l(q)) = 0$) then yields an expression for $\frac{d\phi}{d\vartheta}$
\begin{equation}\label{eqn:app:derivative:volume_constraint}
  \frac{d\phi}{d\vartheta} = -\frac{\sum_{\hat{e} \in \hat{\mathcal{E}}^*} \oneDIntegral{\hat\tau_l}{\hat\tau_r}{\squarepar{\roundpar{\phi - \frac{\kappa}{2} \tau^2} \kappa \tau - \tau}}{\tau}}{\sum_{\hat{e} \in \hat{\mathcal{E}}^*} \oneDIntegral{\hat\tau_l}{\hat\tau_r}{}{\tau}},
\end{equation}
for which each of the integrals can easily be evaluated analytically.

Applying the interface parametrisation to~\cref{eqn:app:derivative:firstmomentder} similarly results in
\begin{align}
  \frac{d}{d\vartheta} \+M_1(c \cap l(q)) &= \sum_{\hat{e} \in \hat{\mathcal{E}}^*} \oneDIntegral{\hat\tau_l}{\hat\tau_r}{\gradient q \cdot \+v_\vartheta \+p(\tau)}{\tau}\\
  &= \sum_{\hat{e} \in \hat{\mathcal{E}}^*} \oneDIntegral{\hat\tau_l}{\hat\tau_r}{\curlypar{\squarepar{\+\eta + \kappa \tau \+\tau} \cdot \squarepar{\roundpar{\frac{d\phi}{d\vartheta}-\tau} \+\eta + \roundpar{\phi - \frac{\kappa}{2} \tau^2} \+\tau}}\\ &\quad\curlypar{\roundpar{\phi - \frac{\kappa}{2} \tau^2} \+\eta + \tau \+\tau}}{\tau}\\
  &= \+\eta \sum_{\hat{e} \in \hat{\mathcal{E}}^*} \oneDIntegral{\hat\tau_l}{\hat\tau_r}{\curlypar{\roundpar{\frac{d\phi}{d\vartheta} - \tau}\roundpar{\phi - \frac{\kappa}{2} \tau^2} + \kappa\tau \roundpar{\phi - \frac{\kappa}{2} \tau^2}^2}}{\tau} \\ &\quad + \+\tau \sum_{\hat{e} \in \hat{\mathcal{E}}^*} \oneDIntegral{\hat\tau_l}{\hat\tau_r}{\curlypar{\frac{d\phi}{d\vartheta} - \tau + \kappa\tau \roundpar{\phi - \frac{\kappa}{2} \tau^2}}\tau}{\tau},
\end{align}
which again can be evaluated analytically after substituting~\cref{eqn:app:derivative:volume_constraint}.

\begin{example}[Linear interface in a convex control volume]
  Note that convexity of the control volume with $\kappa = 0$ implies that the interface consists of a single segment $\hat{e}$.
  The vanishing of the derivative of the zeroth moment then implies that
  \begin{equation}\label{eqn:app:derivative:mof_volcons}
    0 = \frac{d}{d\vartheta} M_0(c \cap l(q)) = \oneDIntegral{\hat\tau_l}{\hat\tau_r}{\roundpar{\frac{d\phi}{d\vartheta} - \tau}}{\tau},
  \end{equation}
  and therefore
  \begin{equation}
    \frac{d\phi}{d\vartheta} = \frac{\oneDIntegral{\hat\tau_l}{\hat\tau_r}{\tau}{\tau}}{\oneDIntegral{\hat\tau_l}{\hat\tau_r}{}{\tau}},
  \end{equation}
  which corresponds to the tangential co-ordinate of the interface midpoint.
  Similarly, the derivative of the first moment is given by
  \begin{align}
    \frac{d}{d\vartheta} \+M_1(c \cap l(q)) &= \oneDIntegral{\hat\tau_l}{\hat\tau_r}{\roundpar{\frac{d\phi}{d\vartheta} - \tau}\roundpar{\phi \+\eta + \tau \+\tau}}{\tau}\\
    &= \+\eta \phi\oneDIntegral{\hat\tau_l}{\hat\tau_r}{\roundpar{\tau - \frac{d\phi}{d\vartheta}}}{\tau} + \+\tau\oneDIntegral{\hat\tau_l}{\hat\tau_r}{\roundpar{\tau - \frac{d\phi}{d\vartheta}}\tau}{\tau},
  \end{align}
  where we have made use of the fact that the interface is linear and therefore the normal $\+\eta$ as well as the tangent $\+\tau$ are constant w.r.t. $\tau$.
  Furthermore, by making use of~\cref{eqn:app:derivative:mof_volcons} we find that the term multiplied by $\+\eta$ vanishes, and therefore
  \begin{equation}
    \frac{d}{d\vartheta} \+M_1(c \cap l(q)) = \+\tau\oneDIntegral{\hat\tau_l}{\hat\tau_r}{\roundpar{\tau - \frac{d\phi}{d\vartheta}}^2}{\tau} = \frac{1}{3} \squarepar{\roundpar{\hat\tau_r - \frac{d\phi}{d\vartheta}}^3 - \roundpar{\hat\tau_l - \frac{d\phi}{d\vartheta}}^3}\+\tau.
  \end{equation}
  Note that since $\frac{d\phi}{d\vartheta}$ corresponds to the tangential co-ordinate of the interface midpoint (for $\kappa = 0$), we find that $\hat\tau_r - \frac{d\phi}{d\vartheta} = \frac{\interfaceLength}{2} = \frac{d\phi}{d\vartheta} - \hat\tau_l$, and therefore
  \begin{equation}
    \frac{d}{d\vartheta} \+M_1(c \cap l(q)) = \frac{1}{3} \squarepar{\roundpar{\frac{\interfaceLength}{2}}^3 - \roundpar{-\frac{\interfaceLength}{2}}^3}\+\tau = \frac{1}{12} \interfaceLength^3 \+\tau,
  \end{equation}
  which agrees with the result from~\citet{Dyadechko2005} in~\cref{eqn:app:derivative:plic}.
\end{example}

\subsection{Derivatives of the LVIRA cost function}\label{sec:app:derivative_l2}
The linearisation of the LVIRA cost function (given by~\cref{eqn:plic:lvira:costfun}) requires the derivative of the zeroth moment of neighbouring control volumes w.r.t. the normal angle $\vartheta$, in the following way
\begin{equation}
  \frac{d}{d\vartheta} f^2_{L^2}(q) = \sum_{c' \in \mathcal{C}(c)} \frac{2}{M_0(c')} \roundpar{M_0(c' \cap \halfspace{q}) - M^l_{0,c'}} \frac{d}{d\vartheta}M_0(c' \cap \halfspace{q}),
\end{equation}
and similarly for the curvature $\kappa$ (only required for PROST).
As with the discussion in~\cref{sec:app:derivative_mof}, the derivative of the shift $\phi$ w.r.t. the normal angle is defined by conservation of volume.
Hence~\cref{eqn:app:derivative:volume_constraint} still holds
\begin{equation}\label{eqn:app:derivative_l2:volume_constraint_s_varphi}
  \frac{d\phi}{d\vartheta} = -\frac{\sum_{\hat{e} \in \hat{\mathcal{E}}_c^*} \oneDIntegral{\hat\tau_l}{\hat\tau_r}{\squarepar{\roundpar{\phi - \frac{\kappa}{2} \tau^2} \kappa \tau - \tau}}{\tau}}{\sum_{\hat{e} \in \hat{\mathcal{E}}_c^*} \oneDIntegral{\hat\tau_l}{\hat\tau_r}{}{\tau}},
\end{equation}
where we now write $\hat{\mathcal{E}}_c^*$ to emphasise that the edges correspond to the control volume $c$: the volume in the centre control volume $c$ is conserved.
For neighbouring control volumes, $c' \in \mathcal{C}(c)$, we find that the derivative of the zeroth moment of intersection is given by (as follows from~\cref{eqn:app:derivative:zeroth_moment_derivative})
\begin{equation}
  \frac{d}{d\vartheta} M_0(c' \cap l(q)) = \sum_{\hat{e} \in \hat{\mathcal{E}}_{c'}^*} \oneDIntegral{\hat\tau_l}{\hat\tau_r}{\squarepar{\frac{d\phi}{d\vartheta} - \tau + \kappa \tau \roundpar{\phi - \frac{\kappa}{2} \tau^2}}}{\tau},
\end{equation}
which can be evaluated analytically.

A similar approach is followed for the derivative of the zeroth moment w.r.t. the curvature.
Analogous to~\cref{eqn:app:derivative:pos_der_varphi} we note that the derivative (w.r.t. the curvature) of a point on the parabola is given by
\begin{equation}
  \+v_\kappa = \frac{d}{d\kappa} \+p = \roundpar{\frac{d\phi}{d\kappa}-\frac{\tau^2}{2}} \+\eta,
\end{equation}
where $\frac{d\phi}{d\kappa}$ is such that the zeroth moment of intersection does not change when changing the curvature.
To this end we compute the derivative of the zeroth moment of intersection of any control volume $c'$ w.r.t. the curvature $\kappa$ (we follow the same steps as in~\cref{eqn:app:derivative:volumeder,eqn:app:derivative:intermediate_result,eqn:app:derivative:explicit_interface_der})
\begin{equation}\label{eqn:app:derivative_l2:volume_der_kappa}
  \frac{d}{d\kappa} M_0(c' \cap l(q)) = \sum_{\hat{e} \in \hat{\mathcal{E}}_{c'}^*} \oneDIntegral{\hat\tau_l}{\hat\tau_r}{\roundpar{\frac{d\phi}{d\kappa} - \frac{\tau^2}{2}}}{\tau}.
\end{equation}
Conservation of the centred volume (i.e. letting $c' = c$ in~\cref{eqn:app:derivative_l2:volume_der_kappa} and equating the result to zero) yields the following expression for the derivative of the shift $\phi$ w.r.t. the curvature $\kappa$
\begin{equation}\label{eqn:app:derivative_l2:volume_constraint_s_kappa}
  \frac{d\phi}{d\kappa} = \frac{\sum_{\hat{e} \in \hat{\mathcal{E}}_{c}^*} \oneDIntegral{\hat\tau_l}{\hat\tau_r}{{\frac{\tau^2}{2}}}{\tau}}{\sum_{\hat{e} \in \hat{\mathcal{E}}_{c}^*} \oneDIntegral{\hat\tau_l}{\hat\tau_r}{}{\tau}}.
\end{equation}
The resulting expression can then be substituted in~\cref{eqn:app:derivative_l2:volume_der_kappa} for the computation of the derivative w.r.t. the curvature of the zeroth moment of intersection.

% \revone{
%   \section{Results of the vortex reverse problem}\label{sec:app:table}
%   \begin{table}[h]
%     \centering
%     \input{tikz/int_validation_advection_table}
%     \caption{Accuracy of the interface shape for the vortex reverse problem at $t = T$.
%       The same error measures as previously displayed in~\cref{fig:int_validation_advection} are used, but we have here also included the accuracy in terms of the $L^1$-norm of the volume fraction error.}
%       \label{tab:int_validation_advection}
%   \end{table}
% }{}

% \section*{References}
\bibliography{library}
\end{document}

\pagenumbering{gobble}
  % \section{Results of the vortex reverse problem}\label{sec:app:table}
  \begin{table}[]
    \centering
    \begin{tabular}{ccccccc} 
\toprule 
\multicolumn{2}{r}{Method:} & ELVIRA & PLVIRA & PROST & MOF & PMOF\\ 
& $h$ &  &  &  &  & \\ 
\midrule 
\multirow{8}{*}{\rotatebox[]{90}{$\|\approximate{\volfrac} - \volfrac\|_{L^1}$}} & $2^{-4}$& $3.6 \times 10^{-3}$& $2.7 \times 10^{-3}$& $2.4 \times 10^{-3}$& $9.1 \times 10^{-4}$& $2.5 \times 10^{-4}$\\ 
 & $2^{-5}$& $9.9 \times 10^{-4}$& $6.0 \times 10^{-4}$& $5.7 \times 10^{-4}$& $2.2 \times 10^{-4}$& $4.1 \times 10^{-5}$\\ 
 & $2^{-6}$& $2.1 \times 10^{-4}$& $9.4 \times 10^{-5}$& $9.0 \times 10^{-5}$& $5.6 \times 10^{-5}$& $6.6 \times 10^{-6}$\\ 
 & $2^{-7}$& $3.5 \times 10^{-5}$& $1.2 \times 10^{-5}$& $1.2 \times 10^{-5}$& $1.1 \times 10^{-5}$& $7.3 \times 10^{-7}$\\ 
 & $2^{-8}$& $6.5 \times 10^{-6}$& $1.6 \times 10^{-6}$& $1.7 \times 10^{-6}$& $2.6 \times 10^{-6}$& $3.2 \times 10^{-7}$\\ 
 & $2^{-9}$& $1.3 \times 10^{-6}$& $2.1 \times 10^{-7}$& $2.1 \times 10^{-7}$& $6.7 \times 10^{-7}$& $3.0 \times 10^{-8}$\\
\bottomrule 
\end{tabular} 
    \caption{Accuracy of the interface shape for the vortex reverse problem at $t = T$.
      The same error measures as previously displayed in fig. 13 are used, but we have here also included the accuracy in terms of the $L^1$-norm of the volume fraction error.}
      \label{tab:int_validation_advection}
  \end{table}